\documentclass{gtpart}
\usepackage{mathrsfs,amscd,subfigure,graphicx,comment}

\graphicspath{{./figures/}}

\specialcomment{Notes}{\begin{quote}\footnotesize}{\end{quote}}
\excludecomment{Notes}

\title{The Jones polynomial of ribbon links}

\author{Michael Eisermann}
\address{Institut Fourier, Universit\'e Grenoble I, France}
\email{Michael.Eisermann@ujf-grenoble.fr}
\urladdr{http://www-fourier.ujf-grenoble.fr/~eiserm}

\date{first version August 2007; this version compiled \today}


\arxivreference{0802.2287}
\arxivpassword{}

\volumenumber{}
\issuenumber{}
\publicationyear{}
\papernumber{}
\startpage{}
\endpage{}
\doi{}
\MR{}
\Zbl{}
\received{}
\revised{}
\accepted{}
\published{}
\publishedonline{}
\proposed{}
\seconded{}
\corresponding{}
\editor{}
\version{}

\theoremstyle{plain}
\newtheorem{Lemma}{Lemma}
\makeautorefname{Lemma}{Lemma}
\newtheorem{Proposition}{Proposition}
\makeautorefname{Proposition}{Proposition}
\newtheorem{Corollary}[Lemma]{Corollary}
\makeautorefname{Corollary}{Corollary}
\newtheorem{Theorem}{Theorem}
\makeautorefname{Theorem}{Theorem}

\newtheorem{theorem}{Theorem}[section]
\newtheorem{lemma}[theorem]{Lemma}
\newtheorem{proposition}[theorem]{Proposition}
\newtheorem{corollary}[theorem]{Corollary}

\theoremstyle{definition}
\newtheorem{definition}[theorem]{Definition}
\newtheorem{question}[theorem]{Question}
\newtheorem{remark}[theorem]{Remark}
\newtheorem{example}[theorem]{Example}
\newtheorem*{notation}{Notation}

\makeautorefname{subsection}{Subsection}
\newcommand{\sref}[1]{\fullref{#1}}

\newcommand{\N}{\mathbb{N}}
\renewcommand{\S}{\mathbb{S}}
\renewcommand{\D}{\mathbb{D}}
\newcommand{\K}{\mathbb{K}}

\newcommand{\Links}{\mathscr{L}}
\newcommand{\Diagrams}{\mathscr{D}}
\newcommand{\BImmersions}{\mathscr{B}}

\newcommand{\id}{\operatorname{id}}
\newcommand{\into}{\hookrightarrow}

\newcommand{\minus}{\smallsetminus}

\newcommand{\lk}{\operatorname{lk}}
\newcommand{\bracket}[1]{\langle #1 \rangle}
\newcommand{\cs}{\mathbin{\sharp}}
\newcommand{\Sign}{\operatorname{sign}}
\newcommand{\Null}{\operatorname{null}}
\newcommand{\ord}{\operatorname{ord}}

\newcommand{\U}[1][N]{U_{#1}}
\newcommand{\V}[1][N]{V_{#1}}
\newcommand{\rV}[1][N]{\tilde{V}_{#1}}
\newcommand{\onehalf}{{\smash{^{_1}\!\!/\!_{^2}}}}
\newcommand{\Horz}{=}
\renewcommand{\Vert}{\parallel}
\newcommand{\Kh}{\mathit{Kh}}
\newcommand{\CKh}{\mathit{CKh}}
\newcommand{\cable}[1]{\smash{#1}}

\begin{document} 

\begin{abstract}
  For every $n$--component ribbon link $L$ we prove 
  that the Jones polynomial $V(L)$ is divisible by 
  the polynomial $V(\bigcirc^n)$ of the trivial link.
  This integrality property allows us to define a generalized determinant
  $\det V(L) := [V(L)/V(\bigcirc^n)]_{(t \mapsto -1)}$, for which 
  we derive congruences reminiscent of the Arf invariant:
  every ribbon link $L = K_1 \cup\dots\cup K_n$ satisfies 
  $\det V(L) \equiv \det(K_1) \cdots \det(K_n)$ modulo $32$,
  whence in particular $\det V(L) \equiv 1$ modulo $8$.

  These results motivate to study the power series expansion
  $V(L) = \sum_{k=0}^\infty d_k(L) h^k$ at $t=-1$, instead of $t=1$ as usual.
  We obtain a family of link invariants $d_k(L)$, 
  starting with the link determinant $d_0(L) = \det(L)$ 
  obtained from a Seifert surface $S$ spanning $L$.
  The invariants $d_k(L)$ are not of finite type with 
  respect to crossing changes of $L$, but they turn out to be 
  of finite type with respect to band crossing changes of $S$. 
  This discovery is the starting point of a theory of surface invariants of finite type,
  which promises to reconcile quantum invariants with the theory 
  of Seifert surfaces, or more generally ribbon surfaces. 
\end{abstract}

\begin{asciiabstract}
  For every n-component ribbon link L we prove 
  that the Jones polynomial V(L) is divisible 
  by the polynomial V(O^n) of the trivial link.
  This integrality property allows us to define a generalized 
  determinant  det V(L) := [V(L)/V(O^n)]_(t=-1), for which 
  we derive congruences reminiscent of the Arf invariant:
  every ribbon link  L = (K_1,...,K_n)  satisfies 
  det  V(L) = det(K_1) ... det(K_n) modulo 32,
  whence in particular  det V(L) = 1 modulo 8.

  These results motivate to study the power series expansion
  V(L) = \sum_{k=0}^\infty d_k(L) h^k at t=-1, instead of t=1 as usual.
  We obtain a family of link invariants d_k(L), 
  starting with the link determinant d_0(L) = det(L) 
  obtained from a Seifert surface S spanning L.
  The invariants d_k(L) are not of finite type with 
  respect to crossing changes of L, but they turn out to be 
  of finite type with respect to band crossing changes of S. 
  This discovery is the starting point of a theory of surface invariants of 
  finite type, which promises to reconcile quantum invariants with the theory 
  of Seifert surfaces, or more generally ribbon surfaces.
\end{asciiabstract}



\maketitle



\section{Introduction} \label{sec:Introduction}

It is often lamented that, after more than $20$ years 
of intense research and spectacular success, we still do not have 
a good topological understanding of the Jones polynomial.  
This is in sharp contrast to the Alexander polynomial:
to mention just one prominent example (Fox--Milnor \cite{FoxMilnor:1966}), 
the Alexander polynomial $\Delta(K)$ of every ribbon or slice knot $K$ has a beautiful 
and very strong symmetry, $\Delta(K) = f(q) \cdot f(q^{-1})$, 
whereas no similar result is known for the Jones polynomial.

Only a few special values of the Jones polynomial have 
a topological interpretation, most notably the determinant 
$\det(L) = V(L)|_{(q \mapsto i)} = \Delta(L)|_{(q \mapsto i)}$. 
(See \sref{sec:Definitions} for definitions; 
we use the parametrization $t = q^2$ throughout.)

\subsection{Statement of results}

We define the nullity $\Null V(L)$ of the Jones polynomial $V(L)$ 
to be the multiplicity of the zero at $q=i$.
The trivial link with $n$ components, for example, 
satisfies $V(\bigcirc^n) = (q^{+1}+q^{-1})^{n-1}$  
and thus $\Null V(\bigcirc^n) = n-1$.

\begin{Lemma} \label{Lem:UpperNullityBound}
  The Jones nullity equals the multiplicity 
  of the factor $(q^{+1}+q^{-1})$ in $V(L)$.
  Every $n$--component link $L$ satisfies $V(L)|_{(q \mapsto 1)} = 2^{n-1}$
  and thus $0 \le \Null V(L) \le n-1$.
\end{Lemma}

This inequality provides a first piece in the puzzle:
the same bounds $0 \le \Null(L) \le n-1$ hold 
for Murasugi's nullity derived from the Seifert form;
see \sref{sec:Definitions} for details. 

\begin{Proposition} \label{Prop:EulerJonesNullity}
  Consider a link $L \subset \R^3$ bounding a properly embedded 
  smooth surface $S \subset \R^4_+$ without local minima.
  If $S$ has positive Euler characteristic $n = \chi(S)$, 
  then the Jones polynomial $V(L)$ is divisible 
  by $V(\bigcirc^n) = (q^+ + q^-)^{n-1}$ and so $\Null V(L) \ge n-1$.
\end{Proposition}

The condition is equivalent to saying that $L$ bounds an immersed 
surface $S \subset \R^3$ of Euler characteristic $n$ and 
having only ribbon singularities; see \sref{sec:RibbonLinks} for details.

Again the same inequality, $\Null(L) \ge \chi(S)-1$, holds for the Seifert nullity.  
Upper and lower bound for $\Null V(L)$ co\"incide precisely for ribbon links:

\begin{Theorem} \label{Thm:RibbonJonesNullity}
  Every $n$--component ribbon link $L$ satisfies $\Null V(L) = n-1$.
\end{Theorem}

This corresponds to the Seifert nullity, so we see 
that $\Null(L) = \Null V(L)$ for every ribbon link $L$.
It would be interesting to know whether this equality 
generalizes to all links, see Questions \ref{quest:BoundaryLinks}
and \ref{quest:ClassicalNullity} towards the end of this article.

Expanding $V(L)$ in $q = \exp(h/2)$ 
we obtain a power series $V(L) = \sum_{k=0}^\infty v_k(L) h^k$ 
whose coefficients $v_k(L)$ are link invariants of finite type 
in the sense of Vassiliev \cite{Vassiliev:1990} and Goussarov \cite{Goussarov:1991},
see also Birman--Lin \cite{BirmanLin:1993} and Bar-Natan \cite{BarNatan:1995}.
The above results motivate to study the power series expansion
$V(L) = \sum_{k=0}^\infty d_k(L) h^k$ in $q = i \exp(h/2)$. 
We obtain a family of invariants $d_k(L)$ 
starting with the link determinant $d_0(L) = \det(L)$.
The Jones nullity $\Null V(L)$ is 
the smallest index $\nu$ such that $d_\nu(L) \ne 0$.
The link invariants $d_k(L)$ are not of finite type with respect to crossing changes.
They enjoy, however, the following surprising property:

\begin{Proposition} \label{Prop:FiniteTypeSurfaceInvariants}
  The surface invariant $S \mapsto d_k(\partial S)$
  is of finite type with respect to band crossing changes.
  More precisely, it is of degree $\le m$ for $m = k + 1 - \chi(S)$.
\end{Proposition}

See \sref{sub:FiniteTypeSurfaceInvariants} for definitions.
The only invariant of degree $<0$ is the zero map:
for $k < \chi(S)-1$ we thus have $d_k(\partial S) = 0$ 
as in \fullref{Prop:EulerJonesNullity}.
The case $m=0$ corresponds to $k = \chi(S)-1$; being of degree $\le 0$
means that $d_k(\partial S)$ is invariant under band crossing changes.
Specializing to ribbon links we obtain the following result:

\begin{Corollary} \label{Cor:BandCrossingChanges}
  For every $n$--component ribbon link $L = K_1 \cup\dots\cup K_n$ 
  the Jones determinant $\det V(L) := [V(L)/V(\bigcirc^n)]_{(q \mapsto i)}$
  is invariant under band crossing changes.  
  If $L$ bounds an immersed ribbon surface $S \subset \R^3$ 
  consisting of $n$ disjoint disks, then we have 
  $\det V(L) = \det(K_1) \cdots \det(K_n)$, whence
  $\det V(L)$ is an odd square integer.
\end{Corollary}


For a ribbon surface $S \subset \R^3$ consisting of disks 
which may intersect each other, multiplicativity only holds 
modulo $32$, and examples show that this is best possible:

\begin{Theorem} \label{Thm:MultiplicativityMod32}
  Every $n$--component ribbon link $L = K_1 \cup\dots\cup K_n$ satisfies 
  $\det V(L) \equiv \det(K_1) \cdots \det(K_n)$ modulo $32$,
  and in particular $\det V(L) \equiv 1$ modulo $8$.
\end{Theorem}

These results can be seen as a first step towards 
understanding the Jones polynomial of ribbon links. 
It is plausible to expect that our results can be extended in several ways, 
and we formulate some natural questions in \sref{sec:OpenQuestions}.
As an application, \fullref{Thm:RibbonJonesNullity} is used 
in \cite{EisermannLamm:SymJones} as an integrality property 
of the Jones polynomial of symmetric unions.

\subsection{Related work} \label{sub:RelatedWork}

Little is known about the Jones polynomial of ribbon knots, 
but there is strong evidence that the expansion 
at $t=-1$ (that is, $q=i$) plays a crucial r\^ole.

First of all, for every ribbon knot $K$,
the determinant $d_0(K) = \det(K) = V(K)_{t \mapsto -1}$
is a square integer, see Remark \ref{rem:AlexanderSlice},
and the resulting congruence $\det(K) \equiv 1 \mod{8}$ 
is related to the Arf invariant of knots,
see Lickorish \cite[chapter 10]{Lickorish:1997}.

Next, the first-order term $d_1(K) = -\bigl[ \frac{d}{dt}V(K) \bigr]_{t \mapsto -1}$ 
figures prominently in the work of Mullins \cite[Theorem 5.1]{Mullins:1993},
who discovered a beautiful relation with the Casson--Walker invariant 
$\lambda(\Sigma^2_K)$ of the $2$-fold branched cover of $\S^3$ branched along $K$:
\begin{equation} \label{eq:Mullins}
  \lambda(\Sigma^2_K) 
  = \frac{1}{4} \Sign(K) - \frac{1}{6}\biggl[\frac{ \frac{d}{dt} V(K) }{ V(K) }\biggr]_{t \mapsto -1}  
  = \frac{1}{4} \Sign(K) + \frac{1}{6} \frac{ d_1(K) }{ d_0(K) }. 
\end{equation}

This identity holds for every knot $K \subset \S^3$, 
and more generally for every link with non-vanishing determinant.
Garoufalidis \cite[Theorem 1.1]{Garoufalidis:1999} generalized Mullins' result 
to all links, using the Casson--Walker--Lescop invariant $\Lambda$ 
constructed by Lescop \cite{Lescop:1996}: 
\begin{equation} \label{eq:Garoufalidis}
  i^{ [ \Sign(K) + \Null(K) ] } \cdot \Lambda(\Sigma^2_K) 
  = \frac{1}{4} d_0(K) \Sign(K) + \frac{1}{6} d_1(K) . 
\end{equation}


If $\det(K) \ne 0$, then $\Sigma^2_K$ is a rational homology sphere;
in this case the Casson--Walker invariant is defined and satisfies 
$\lambda(\Sigma^2_K) \cdot \det(K) = \Lambda(\Sigma^2_K)$,
so that \eqref{eq:Garoufalidis} implies \eqref{eq:Mullins}.

If $\det(K)=1$, then $\Sigma^2_K$ is an integral homology sphere
and $\lambda(\Sigma^2_K) \in \Z$ is Casson's original invariant.
If, moreover, $K$ is a ribbon knot, then $\Sign(K)$ vanishes
and $\lambda(\Sigma^2_K)$ is an even integer because 
it reduces modulo $2$ to the Rohlin invariant and $\Sigma^2_K$ bounds 
a homology $4$--ball, see Casson--Gordon \cite[Lemma 2]{CassonGordon:1986}.
In this case $d_1(K)$ is divisible by $12$.
No such congruences seem to be known for higher order terms
$d_2, d_3, \dots$, nor for ribbon knots or links in general.

Generalizing work of Sakai, Mizuma has worked out an explicit formula 
for $d_1(K)$ of $1$--fusion ribbon knots $K$ \cite{Mizuma:2005} 
and derived a lower bound for the ribbon number \cite{Mizuma:2006}.

Studying link concordance, Cochran \cite[Corollary 3.10]{Cochran:1985} 
has established similar properties and congruences for the first 
non-vanishing coefficients of the Conway polynomial.

\subsection{How this article is organized}

Theorems \ref{Thm:RibbonJonesNullity} and \ref{Thm:MultiplicativityMod32} 
are pleasant to state but their proofs are somewhat technical:  
we proceed by induction on planar diagrams of immersed surfaces in $\R^3$.
The arguments are elementary but get increasingly entangled.
Generally speaking, these technicalities are due to 
the combinatorial definition of the Jones polynomial 
whereas the ribbon condition is topological in nature.

The article follows the outline given in this introduction.
\fullref{sec:Definitions} recollects some basic definitions and 
highlights motivating analogies; the upper bound 
of \fullref{Lem:UpperNullityBound} is derived from 
Jones' skein relation by an algebraic argument.  In order to
apply skein relations to ribbon links, \fullref{sec:RibbonLinks} 
recalls the notions of slice and ribbon links, and introduces 
planar band diagrams as a convenient presentation.
\fullref{sec:RibbonJonesNullity} sets up a suitable 
induction technique for the Kauffman bracket and proves 
the lower bound of \fullref{Prop:EulerJonesNullity}.
\fullref{sec:BandCrossingChanges} discusses band crossing changes
and proves \fullref{Prop:FiniteTypeSurfaceInvariants}.
\fullref{sec:JonesDeterminant} establishes multiplicativity 
modulo $32$ as stated in \fullref{Thm:MultiplicativityMod32}.
\fullref{sec:OpenQuestions}, finally, discusses 
possible generalizations and open questions.

\subsection{Acknowledgements}

I thank Christoph Lamm for countless inspiring discussions, 
which sparked off the subject.  I am also indebted to 
Christine Lescop, Thomas Fiedler and Christian Blanchet for 
valuable suggestions and comments on successive versions of this article.
This work was begun in the winter term 2006/2007 
during a sabbatical funded by a research contract
\textit{d\'el\'egation aupr\`es du CNRS}, 
whose support is gratefully acknowledged.


\section{Definitions and first properties} \label{sec:Definitions}

The nullity and the determinant that we introduce and study 
for the Jones polynomial are analogous to the corresponding notions 
of the classical Seifert form.  In order to highlight these intriguing 
analogies most convincingly, we shall review side by side some elementary 
properties of the Alexander--Conway and the Jones polynomial.

As standard references in knot theory we refer to
Burde--Zieschang \cite{BurdeZieschang:1985},
Lickorish \cite{Lickorish:1997} and Rolfsen \cite{Rolfsen:1990}.
Throughout this article we work in the smooth category.

\subsection{The Alexander--Conway polynomial} \label{sub:AlexanderConwayPolynomial}

We denote by $\Z[q^{\pm}]$ the ring of Laurent polynomials 
in the variable $q = q^+$ with inverse $q^{-1} = q^-$.
Its elements will simply be called \emph{polynomials} in $q$. 
%
%
For every link $L \subset \R^3$ we can construct a Seifert surface $S$ 
spanning $L$, that is, a compact connected oriented surface $S \subset \R^3$ 
such that $L = \partial S$ with induced orientations.
We choose a basis of $H_1(S)$ and denote by $\theta$ and $\theta^*$ 
the associated Seifert matrix and its transpose, respectively;
see Burde--Zieschang \cite[Definition 8.5]{BurdeZieschang:1985},
Lickorish \cite[Definition 6.5]{Lickorish:1997} or Rolfsen \cite[Definition 8A1]{Rolfsen:1990}.
The Alexander--Conway polynomial $\Delta(L) \in \Z[q^{\pm}]$ 
is defined as $\Delta(L) = \det( q^- \, \theta^* - q^+ \, \theta )$.
It does not depend on the choice of $S$ and 
is thus an isotopy invariant of the link $L$.
It is traditionally parametrized by $t = q^2$, but 
we prefer the variable $q = -t^{\onehalf}$ in order to 
avoid square roots and to fix signs.

We denote by $\Links$ the set of isotopy classes of oriented links $L \subset \R^3$.
The map $\Delta \co \Links \to \Z[q^{\pm}]$ is characterized by Conway's skein relation 
\[
\newcommand{\skein}[1]{\Delta{\bigl(\raisebox{-0.9ex}{\includegraphics[height=3ex]{#1}}\bigr)}}
\skein{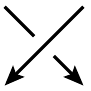} - \skein{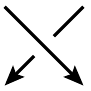} = ( q^+ - q^- ) \skein{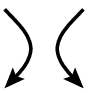} 
\]
with the initial value $\Delta(\bigcirc) = 1$.
The skein relation entails that $\Delta(L \sqcup \bigcirc) = 0$.



The (signed) \emph{determinant} of a link is defined as 
\[
\det(L) := \Delta(L)|_{(q \mapsto i)} = \det[ -i (\theta+\theta^*) ].
\]
Most authors consider the determinant $\det(\theta+\theta^*)$, 
but then only its absolute value $|\det(\theta+\theta^*)|$ is invariant;
see Burde--Zieschang \cite[Corollary 13.29]{BurdeZieschang:1985}, 
Lickorish \cite[p.\,90]{Lickorish:1997}, Rolfsen \cite[Definition 8D4]{Rolfsen:1990}.
In our normalization the determinant is the unique invariant $\det \co \Links \to \Z[i]$ 
that satisfies $\det(\bigcirc) = 1$ and the skein relation
\[
\newcommand{\skein}[1]{\det{\bigl(\raisebox{-0.9ex}{\includegraphics[height=3ex]{#1}}\bigr)}}
\skein{skein+} - \skein{skein-} = 2i \skein{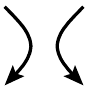} .
\]

\subsection{Signature and nullity} \label{sub:SignatureNullity}

Murasugi \cite{Murasugi:1965} showed that 
the \emph{signature} $\Sign(L) := \Sign(\theta+\theta^*)$ 
and the \emph{nullity} $\Null(L) := \Null(\theta+\theta^*)$ 
are invariants of the link $L$.
Tristram \cite{Tristram:1969} generalized this by passing
from the symmetric matrix $\theta+\theta^*$ to the hermitian matrix 
\[
M_\omega := (1-\omega^2)\theta + (1-\bar\omega^2)\theta^*
= (\omega-\bar\omega)(\bar\omega\theta^*-\omega\theta) 
\]
with $\omega \in \S^1 \minus \{\pm 1\}$. 
He showed that the \emph{generalized signature} $\Sign_\omega L := \Sign M_\omega$ 
and the \emph{generalized nullity} $\Null_\omega L := \Null M_\omega$
are again link invariants.  Independently, 
Levine \cite{Levine:1969} defined the same invariants for knots.
For $\omega = \pm i$ this specializes to Murasugi's invariants.
For higher dimensions see Erle \cite{Erle:1969} and Milnor \cite{Milnor:1968}.

\begin{remark} \label{rem:Nullity}
  For every knot $K$ we have $\det(K) \equiv 1 \mod{4}$, 
  whence $\det(K) \ne 0$ and $\Null(K) = 0$.  
  More generally, let $N \in \N$ be a prime number 
  and let $\omega$ be a primitive $2N$th root of unity.
  Tristram \cite[Lemma 2.5]{Tristram:1969} 
  remarked that the generalized determinant
  \[
  \textstyle\det_\omega (L) := \Delta(L)|_{(q \mapsto \omega)}
  = \det( \bar\omega \theta^* - \omega \theta )
  \]
  never vanishes for a knot. 
  More generally, he proved that $0 \le \Null_\omega(L) \le n-1$ 
  for every link $L$ with $n$ components \cite[Corollary 2.24]{Tristram:1969}.
  We shall see below that the same technique applies to the Jones polynomial.

  For the matrix $M := q^- \theta^* - q^+ \theta$ over $\Z[q^{\pm}] \subset \Q(q)$,
  the nullity $\Null M$ is the dimension of its null-space. 
  We have $\Null M \le \Null M_\omega$ for all $\omega \in \S^1$,
  and equality holds for all but finitely many values of $\omega$.
  In particular we see that $0 \le \Null M \le n-1$.
\end{remark}

\subsection{The Jones polynomial} \label{sub:JonesPolynomial}

The following theorem is due to Alexander \cite{Alexander:1928} and 
Conway \cite{Conway:1969} for $N=0$, Jones \cite{Jones:1985} for $N=2$, 
and HOMFLYPT 
(Freyd--Yetter--Hoste--Lickorish--Millett--Ocneanu \cite{HOMFLY:1986},
Przytycki--Traczyk \cite{PrzytyckiTraczyk:1988})
for the general case $N \in \N$.

\begin{theorem} 
  \label{thm:HOMFLYPT}
  For each $N \in \N$ there exists a unique link invariant 
  $\V \co \Links \to \Z[q^{\pm}]$ mapping the trivial knot
  to $\V(\bigcirc) = 1$ and satisfying the following skein relation:
  \begin{equation}
    \label{eq:HOMFLYPT}
    \newcommand{\skein}[1]{\V{\bigl(\raisebox{-0.9ex}{\includegraphics[height=3ex]{#1}}\bigr)}}
    q^{-N} \skein{skein+} - q^{+N} \skein{skein-} = (q^{-1}-q^{+1}) \skein{skein0} .
  \end{equation}
\end{theorem}

The case $N=0$ reproduces the Alexander--Conway polynomial, 
$V_0(L) = (-1)^{n-1} \Delta(L)$ where $n$ is the number of components.
The choice $N=1$ yields the trivial invariant, $\V[1](L) = 1$ for all $L \in \Links$.
The case $N=2$ yields the Jones polynomial \cite{Jones:1985}, traditionally 
parametrized by $t = q^2$ with the sign convention $q = -t^{\onehalf}$
(see \sref{sub:KauffmanBracket}).

\begin{remark}
  It follows from these axioms that $\V(L \sqcup \bigcirc) = \V(L) \cdot \U$ with 
  \[
  \U = \frac{q^{-N}-q^{+N}}{q^{-1}-q^{+1}} .
  \]
  We have $\U[0] = 0$ and $\U[1] = 1$, while 
  for $N \ge 2$ we obtain the expansion
  \[
  \U = q^{-N+1} + q^{-N+3} + \dots + q^{N-3} + q^{N-1} .
  \]
  This is sometimes called the \emph{quantum integer} $[N]_q$. 
  For $q \mapsto 1$ we get $[N]_{(q \mapsto 1)} = N$. 

  For the trivial $n$--component link we have $\V(\bigcirc^n) = \U^{n-1}$,
  and for every $n$--component link $L$ we obtain
  \[
  \V(L)|_{(q \mapsto 1)} = \V(\bigcirc^n)|_{(q \mapsto 1)} = N^{n-1} .
  \]
  Finally, we observe the following symmetry 
  with respect to the automorphism $q \mapsto -q$,
  which corresponds to the non-trivial Galois automorphism
  of $\Z[t^{\pm\onehalf}]$ over $\Z[t^\pm]$:
  \[
  \V(L)|_{(q \mapsto -q)} = (-1)^{(N-1)(n-1)} \V(L) .
  \]
  If $N$ is odd, then $\V(L)$ is even, 
  that is, invariant under the automorphism $q \mapsto -q$.
  \linebreak[3]
  If $N$ is even, then $\V(L)|_{(q \mapsto -q)} = (-1)^{n-1} \V(L)$
  depends on the parity of $n$.
\end{remark}

\subsection{An upper bound for the Jones nullity} \label{sub:NullityUpperBound}

We are now ready to prove Lemma \ref{Lem:UpperNullityBound}.
The idea is to adapt Tristram's observation 
\cite[Lemma 2.5]{Tristram:1969} to the Jones polynomial.

\begin{definition}
  The \emph{nullity} $\Null_z P = \nu$ of a Laurent polynomial $P \in \C[q^\pm]$ 
  at some point $z \in \C\minus\{0\}$ is the multiplicity $\nu$ of the root at $q = z$.
\end{definition}

More explicitly, we have $P = (q-z)^\nu \cdot Q$ 
such that $\nu \ge 0$ and $Q \in \C[q^\pm]$ satisfies $Q(z) \ne 0$.
Alternatively, $\nu$ is the least integer such that the 
derivative $P^{(\nu)} = \smash{\frac{d^v}{dq^v}} P$ does not vanish in $z$.
It is also the smallest index such that $d_\nu \ne 0$ in the power series 
expansion $P(q) = \sum_{k=0}^\infty d_k h^k$ at $q = z \exp(h/2)$. 

The polynomial $\U \in \Z[q^\pm]$ of degree $2N-2$ vanishes at 
every $2N$th root of unity $\omega$ other than $\pm 1$, so that $\Null_\omega \U = 1$.  
We fix a primitive $2N$th root of unity, $\omega = \exp( i \pi k/N )$,
by specifying an integer $k$ such that $0 < k < 2N$ and $\gcd(k,2N)=1$.

\begin{proposition}
  Let $L \subset \R^3$ be a link.  
  If $N$ is prime, then we have the factorization $\V(L) = \U^\nu \cdot \rV(L)$ 
  with $\nu = \Null_\omega \V(L)$ and $\rV(L) \in \Z[q^\pm]$ such that $\U \nmid \rV(L)$.
\end{proposition}

\begin{proof}
  If $N=2$, then $\omega = \pm i$, whence $q^2 + 1 = q \U[2]$ 
  is the minimal polynomial of $\omega$ in $\Q[q]$.
  For each $P \in \Z[q^\pm]$, the condition $P(\omega) = 0$ 
  is equivalent to $P = \U[2] \cdot Q$ for some $Q \in \Z[q^\pm]$.
  Iterating this argument, we obtain $P = \U[2]^\nu \cdot Q$ 
  with $Q \in \Z[q^\pm]$ such that $Q(\omega) \ne 0$, 
  whence $\Null_\omega P = \nu$.

  If $N$ is odd, then $\omega$ is of order $2N$ 
  and $-\omega$ is of order $N$ in the multiplicative group $\C^\times$.  
  Their minimal polynomials in $\Q[q]$ are the cyclotomic polynomials
  \cite[\textsection VI.3]{Lang:2002}
  \[
  \Phi_{2N} = \prod_{\begin{smallmatrix} \xi \in \C^\times \\ \ord(\xi)=2N \end{smallmatrix}} (q-\xi)
  \qquad\text{and}\qquad
  \Phi_{N} = \prod_{\begin{smallmatrix} \xi \in \C^\times \\ \ord(\xi)=N \end{smallmatrix}} (q-\xi) .
  \]
  
  If moreover $N$ is prime, then all $2N$th roots of unity 
  are either of order $1$, $2$, $N$, or $2N$ and thus 
  $q^{2N}-1 = (q-1) (q+1) \Phi_{N} \Phi_{2N}$.
  This implies that 
  \[
  \Phi_{N} \cdot \Phi_{2N} = q^{2N-2} + q^{2N-4} + \dots + q^2 + 1 = q^{N-1} \U .
  \]
  This polynomial is even, has integer coefficients and leading coefficient $1$. 
  As a consequence, if $P \in \Z[q^\pm]$ is even, then $P(\omega)=0$ is 
  equivalent to $P = \U \cdot Q$ for some $Q \in \Z[q^\pm]$, and $Q$ is again even.
  Iterating this argument, we obtain $P = \U^\nu \cdot Q$ 
  with $Q \in \Z[q^\pm]$ even and $Q(\omega) \ne 0$, whence $\Null_\omega P = \nu$.
\end{proof}

\begin{corollary} \label{cor:NullityBound}
  Let $N$ be a prime and let $\omega \ne \pm 1$ be a $2N$th root of unity.
  Then the nullity $\Null_\omega \V(L)$ only depends on $N$ 
  and will thus be denoted by $\Null \V(L)$.
  For every link $L$ with $n$ components we have 
  the inequality $0 \le \Null \V(L) \le n-1$.
\end{corollary}

\begin{proof}
  We have $\V(L) = \U^\nu \cdot \rV(L)$ with $\nu = \Null_\omega \V(L)$
  and $\rV(L) \in \Z[q^\pm]$. 
  Evaluating at $q=1$, we find $N^{n-1} = N^\nu \cdot \rV(L)|_{(q \mapsto 1)}$,
  whence $\nu \le n-1$.  
\end{proof}

\begin{definition}
  In the notation of the previous proposition,
  we call $\rV(L)$ the \emph{reduced Jones polynomial} 
  and $\det_\omega \V(L) := \rV(L)|_{(q \mapsto \omega)}$
  the \emph{Jones determinant} of $L$ at $\omega$.
  It depends on the chosen root of unity $\omega$ up to 
  a Galois automorphism of the ring $\Z[\omega]$.
\end{definition}

\begin{remark}
  The family of invariants $\V$ with $N \in \N$ can be encoded by 
  the \textsc{Homflypt} polynomial $P \co \Links \to \Z(q,\ell)$
  defined by $P(\bigcirc)=1$ and the skein relation
  \[
  \newcommand{\skein}[1]{P{\bigl(\raisebox{-0.9ex}{\includegraphics[height=3ex]{#1}}\bigr)}}
  \ell^- \skein{skein+} - \ell^+ \skein{skein-} = (q^- - q^+) \skein{skein0} .
  \]
  This implies that $P(L\sqcup\bigcirc) = P(L) \cdot U$ 
  with $U = \frac{\ell^- - \ell^+}{q^- - q^+}$.
  Moreover, $P$ takes values in the subring $R := \Z[q^\pm,\ell^\pm,U]$
  and is invariant under the ring automorphism $(\ell \mapsto -\ell, q \mapsto -q)$.

  By construction, the following diagram is commutative:
  \[
  \begin{CD}
    \text{\textsc{Homflypt}} @. 
    \Z[ q^\pm, \ell^\pm, U ] @>{\ell \mapsto q^N}>> \Z[q^\pm] 
    @. \text{Jones}
    \\
    @. @V{\ell \mapsto -1}VV @VV{q \mapsto \omega}V @. 
    \\
    \text{Alexander--Conway} @.
    \Z[q^\pm] @>>{q \mapsto \omega}> \C 
    @. \text{determinant}
  \end{CD}
  \]

  For every link $L$ we have a unique factorization $P(L) = U^\nu \cdot Q$ 
  with $\nu \ge 0$ and $Q \in R$ satisfying $Q|_{\ell = \pm 1} \ne 0$.
  We call $\Null P(L) := \nu$ the \emph{nullity} of the \textsc{Homflypt} polynomial.
  It satisfies the inequality $\Null P(L) \le \Null \V(L)$ for all $N \in \N$, 
  and equality holds for all but finitely many values of $N$.
  In particular $0 \le \Null P(L) \le n-1$.
\end{remark}


\section{Band diagrams for ribbon links} \label{sec:RibbonLinks}

\subsection{Band diagrams} \label{sub:BandDiagrams}

We wish to apply skein relations to ribbon links.
To this end we shall use planar \emph{band diagrams} 
built up from the pieces shown in \fullref{fig:BandPieces}.

\begin{figure}[h]
  \centering
  \subfigure[ends, strip, twists]{\quad\includegraphics[scale=0.9]{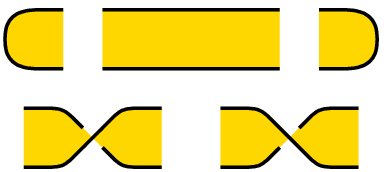}\quad}
  \subfigure[band junction]{\qquad\includegraphics[scale=0.9]{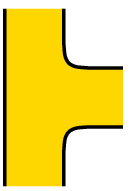}\quad}
  \subfigure[band crossing]{\quad\includegraphics[scale=0.9]{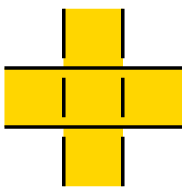}\quad}
  \subfigure[ribbon singularity]{\qquad\includegraphics[scale=0.9]{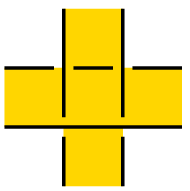}\qquad}
  \caption{Elementary pieces of band diagrams}
  \label{fig:BandPieces}
\end{figure}

Such a diagram encodes not only a link $L \subset \R^3$ 
but also an immersed surface $S \subset \R^3$ 
with boundary $\partial S = L$.  More explicitly:

\begin{definition} \label{def:ImmersedBandSurface}
  Let $\Sigma$ be a smooth compact surface
  with boundary $\partial\Sigma \ne \emptyset$.
  We do not require $S$ to be orientable nor connected, 
  but we will assume that $S$ does not have any closed components.  
  A smooth immersion $f \co \Sigma \looparrowright \R^3$ 
  is called (immersed) \emph{ribbon surface} if its only singularities 
  are ribbon singularities according to the local model shown in \fullref{fig:BandPieces}d.  
  \fullref{fig:RibbonExamples}a displays a more three-dimensional view:
  every component of self-intersection is an arc $A$ 
  so that its preimage $f^{-1}(A)$ consists of two arcs 
  in $\Sigma$, one of which is interior.
\end{definition}

\begin{figure}[h]
  \centering
  \subfigure[a ribbon singularity]{\quad\includegraphics[height=14ex]{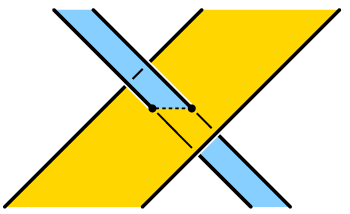}\quad}
  \subfigure[a ribbon link $L$]{\quad\includegraphics[height=14ex]{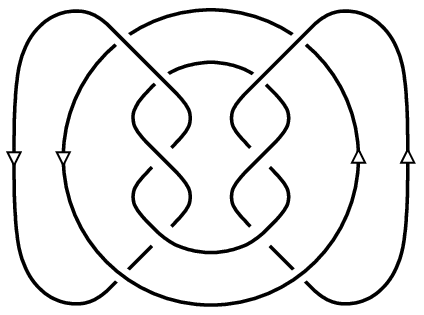}\quad}
  \subfigure[ribbon disks for $L$]{\quad\includegraphics[height=14ex]{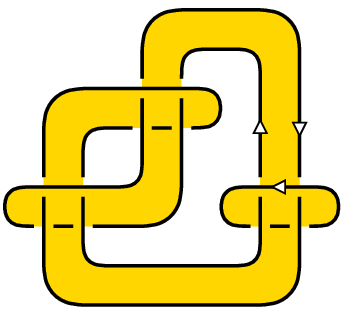}\quad}
  \caption{Ribbon links and immersed ribbon surfaces}
  \label{fig:RibbonExamples}
\end{figure}

\begin{notation}
  We regard $f$ only up to diffeomorphisms of $\Sigma$, 
  and can thus identify the immersion $f$ with its image $S = f(\Sigma)$.
  The \emph{Euler characteristic} $\chi(S)$ of the immersed surface $S$ 
  is  by definition the Euler characteristic of the abstract surface $\Sigma$.
  A \emph{component} of $S$ is the image of a component of $\Sigma$.

  A ribbon singularity is called \emph{mixed} if it involves 
  two distinct surface components.  Otherwise, if the surface 
  component pierces itself, the ribbon singularity is called \emph{pure}.

  We write $S = S_1 \sqcup\dots\sqcup S_n$ if the components 
  $S_1,\dots,S_n$ are contained in disjoint balls in $\R^3$.  
  We also use the analogous notation $L = L_1 \sqcup\dots\sqcup L_n$ for links.
\end{notation}

Since each surface component $S_k$ has non-empty boundary, it satisfies $\chi(S_k) \le 1$.
As a consequence, if a link $L$ has $n$ components, then every 
ribbon surface $S$ spanning $L$ satisfies $\chi(S) \le n$.
The maximum is attained precisely for ribbon links:

\begin{definition}
  An $n$--component link $L \subset \R^3$ 
  is called a \emph{ribbon link} if it bounds 
  a ribbon surface $S \subset \R^3$ consisting of $n$ disks.
  (\fullref{fig:RibbonExamples} shows an example.)
\end{definition}

\begin{proposition}
  For every band diagram $D$ there exists a ribbon surface 
  $S \subset \R^3$ such that the standard projection $\R^3 \to \R^2$ 
  maps $S$ to $D$ in the obvious way.  

  Any two ribbon surfaces projecting to $D$ are ambient isotopic in $\R^3$.  
  Modulo ambient isotopy we can thus speak of \emph{the} surface realizing $D$,

  Every ribbon surface $S \subset \R^3$ can be represented 
  by a band diagram $D$, that is, $S$ is ambient
  isotopic to a surface $S'$ projecting to $D$.
\end{proposition}

\begin{proof} 
  We only sketch the last assertion:
  existence of a band diagram for every ribbon surface $S$.
  The idea is to cut $S$ along properly embedded arcs running from boundary to boundary.
  This is possible under our hypothesis that $S$ has no closed components.
  We repeat this process so as to obtain trivial pieces homeomorphic to disks.  

  \begin{figure}[hbtp]
    \centering
    \includegraphics[height=14ex]{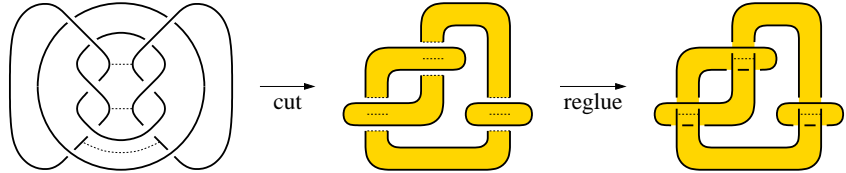}
    \caption{Cutting and regluing a ribbon surface}
    \label{fig:RibbonCutAndPaste}
  \end{figure}
  
  These disks can be put disjointly into the plane and then reglued as prescribed;
  \fullref{fig:RibbonCutAndPaste} illustrates an example.
  Regluing typically introduces junctions, band crossings, twists 
  and ribbon singularities; these suffice to achieve the reconstruction.
\end{proof}

\begin{remark}
  There is an analogue of Reidemeister's theorem,
  representing ambient isotopy of ribbon surfaces $S \subset \R^3$
  by a generating set of local moves on band diagrams $D \subset \R^2$.
  The local moves are straightforward but lengthy to enumerate,
  and we shall not need this more precise result here.
  The general philosophy is that of \emph{links with extra structure}, 
  in our case links with a ribbon surface. 
  This is an interesting topic in its own right, but we shall use
  it here merely as an auxiliary tool for our induction proof.
\end{remark}

\subsection{Slice and ribbon links} \label{sub:SliceRibbonLinks}

In order to put our subject matter into perspective, 
we briefly recall the $4$-dimensional setting of slice and ribbon links.

We consider $\R^3$ as a subset of $\R^4$ via the 
standard inclusion $(x_1,x_2,x_3) \mapsto (x_1,x_2,x_3,0)$.
We say that a link $L \subset \R^3$ bounds a surface 
$S \subset \R^4_+ = \{ x \in \R^4 \mid x_4 \ge 0 \}$
if $S$ is properly embedded smooth surface such that $\partial S = S \cap \R^3 = L$.  
(We will always assume that $S$ has no closed components.)
A link $L$ is called \emph{slice} 
if it bounds $n$ disjointly embedded disks in $\R^4_+$.
This is sometimes called slice \emph{in the strong sense}; 
see Fox \cite{Fox:1962} for a discussion of weaker notions.
Slice knots naturally appear in the study of surfaces $\Sigma \subset \S^4$ 
with singularities, that is, isolated points where the surface $\Sigma$
is not locally flat, see Fox--Milnor \cite{FoxMilnor:1966},
and Livingston \cite{Livingston:2005} for a survey.

If an $n$--component link bounds a surface $S \subset \R^4_+$,
then the Euler characteristic is bounded by $\chi(S) \le n$, 
and the maximum is attained precisely for slice links.
We are particularly interested in the case where 
the surface $S \subset \R^4_+$ has no local minima, 
or more explicitly, the height function $h \co \R^4_+ \to \R_+$, $x \mapsto x_4$, 
restricts to a Morse function $h|_S$ without local minima.
The following observation goes back to Fox \cite{Fox:1962}:

\begin{proposition}
  For every link $L \subset \R^3$ 
  the following assertions are equivalent:
  \begin{enumerate}
  \item
    $L$ bounds an immersed ribbon surface 
    $S \subset \R^3$ such that $\chi(S) = n$.
  \item
    $L$ bounds a surface $S_+ \subset \R^4_+$ 
    without local minima such that $\chi(S_+) = n$.
    \qed
  \end{enumerate}
\end{proposition}

\begin{Notes}
\begin{proof}[Sketch of proof]
  ``$\Rightarrow$''
  Given an immersed ribbon surface $S \subset \R^3$,
  its interior can be pushed into $\R^4_+$ so as to yield 
  a properly embedded surface $S_+ \subset \R^4_+$ without local minima.
  More explicitly, we parametrize $S$ by a smooth 
  immersion $f \co \Sigma \looparrowright \R^3$. 
  This defines a Riemannian metric on $\Sigma$, and we denote by 
  $d \co \Sigma \times \Sigma \to \R_+$ the induced path metric.  
  For each singularity we can assume that the interior arc is
  entirely contained in some metric disk disjoint from the boundary.
  The function $h_0 \co \Sigma \to \R_+$, $h_0(x) = d(x,\partial\Sigma)$
  is continuous, has no local minima on its interior, 
  and satisfies $h_0^{-1}(0) = \partial\Sigma$.
  Moreover, $h_0$ is smooth on some open neighbourhood of the boundary $\partial\Sigma$.  
  By  an arbitrarily small perturbation we obtain a smooth Morse function $h$,
  still without local minima and satisfying $h^{-1}(0) = \partial\Sigma$.  
  The map $F \co \Sigma \to \R^4_+$, $F(x) = (f(x),h(x))$ 
  is a smooth embedding as desired.

  ``$\Leftarrow$''
  Given a smoothly embedded surface $S_+ \subset \R^4_+$, we can put it
  into Morse position with respect to $h$ by an arbitrarily small isotopy.
  Let us assume that critical points occur at the levels $1-\onehalf, \dots, m-\onehalf$.
  We then proceed from top to bottom: for each regular value $t=m,\dots,0$
  we equip the link $L_t = h|_S^{-1}(t)$ with an immersed surface $S_t$
  such that $\chi(S_t) = \chi( h|_S^{-1}([t,m])$.
  Each maximum adds a small embedded disk $D \subset \R^3$
  and increases the Euler characteristic by one, $\chi(S_t) = \chi(S_{t+1})+1$.
  Each saddle point glues a strip to the existing surface $S_{t+1}$
  and decreases the Euler characteristic by one, $\chi(S_t) = \chi(S_{t+1})-1$.
  This strip may pierce the surface, thus creating ribbon singularities.
\end{proof}
\end{Notes}

As a consequence, $L$ is a ribbon link if and only if it bounds 
$n$ disjointly embedded disks in $\R^4_+$ without local minima.
Whether all slice knots are ribbon is an open question 
which first appeared as Problem 25 in Fox's problem list \cite{Fox:1962}.
Also see Problem 1.33 of Kirby's problem list \cite{Kirby:1997}.

\begin{remark} \label{rem:AlexanderSlice}
  Every $n$--component slice link $L$ satisfies 
  $\Null(L) = n-1$ and $\Sign(L) = 0$ (Murasugi \cite{Murasugi:1965}),
  and more generally $\Null_\omega(L) = n-1$ and $\Sign_\omega(L) = 0$ where 
  $\omega$ is a $2N$th root of unity and $N$ is prime (Tristram \cite{Tristram:1969}). 
  For $n = 1$ the Alexander polynomial factors as $\Delta(L) = f(q^+) \cdot f(q^-)$ 
  with some $f \in \Z[q^\pm]$ (Fox--Milnor \cite{FoxMilnor:1966}).
  As a consequence $\det(K)$ is a square integer for 
  every slice knot $K$, in particular $\det(K) \equiv 1 \mod{8}$.
  For $n \ge 2$, however, we have $\Delta(L) = 0$, see Kawauchi \cite{Kawauchi:1978}.
  It is a classical topic to study higher-order 
  Alexander polynomials to remedy this problem; 
  for the multi-variable Alexander polynomial 
  see Kawauchi \cite{Kawauchi:1978} and Florens \cite{Florens:2004}. 
  We will instead look for extensions and analogies 
  in the realm of quantum invariants. 
\end{remark}


\section{The Jones nullity of ribbon links} \label{sec:RibbonJonesNullity}

Jones' skein relation \eqref{eq:HOMFLYPT} serves well for 
the upper nullity bound, but it turns out to be ill suited 
for the inductive proof that we shall be giving for the lower bound.
We will thus prepare the scene by recalling 
Kauffman's bracket (\sref{sub:KauffmanBracket}).  
Ribbon link diagrams suggest a proof by induction, but one has
to suitably generalize the statement (\sref{sub:ProofStrategy}).
I present here what I believe is the simplest induction proof, 
based on the Euler characteristic (\sref{sub:EulerJonesNullity}).

\subsection{The Kauffman bracket} \label{sub:KauffmanBracket}

The Kauffman bracket \cite{Kauffman:1987} is a map $\Diagrams \to \Z[A^{\pm}]$,
denoted by $D \mapsto \bracket{D}$, from the set $\Diagrams$ 
of unoriented planar link diagrams to the ring $\Z[A^{\pm}]$
of Laurent polynomials in the variable $A$.
It is defined by the skein relation
\newcommand{\kskein}[1]{\bigl\langle\raisebox{-0.6ex}{\includegraphics[height=2.5ex]{#1}}\bigr\rangle}
\begin{align*}
  \kskein{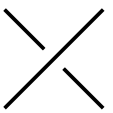} & = A \kskein{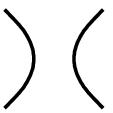} + A^{-1} \kskein{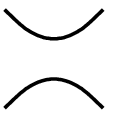} ,
  \\
  \bracket{D \sqcup \bigcirc} & =  \bracket{D} \cdot (-A^{+2}-A^{-2}) ,
  \\
  \bracket{\bigcirc} & = 1.
\end{align*}

The bracket polynomial $\bracket{D}$ is invariant under 
Reidemeister moves R2 and R3, called \emph{regular isotopy}.
%
%
Normalizing with respect to the writhe one obtains
an isotopy invariant:  upon the change of variables 
$q = -A^{-2}$ we thus recover the Jones polynomial 
\[
V(L)|_{(q=-A^{-2})} = \bracket{D} \cdot (-A^{-3})^{\mathrm{writhe}(D)} .
\]
Here $V(L) = \V[2](L)$ is the Jones polynomial of the oriented link $L$, 
while $D$ is a planar diagram representing $L$, and $\bracket{D}$ 
is its bracket polynomial (forgetting the orientation of $D$).
The writhe of $D$ is the sum of all crossing signs.

\begin{notation}
  All subsequent calculations take place in the ring 
  $\Z[A^\pm] \supset \Z[q^\pm] \supset \Z[t^\pm]$
  with $q = -A^{-2}$ and $t = q^2 = A^{-4}$.
  This context explains the sign in $q = -t^{\onehalf}$ with $t^{\onehalf} = A^{-2}$:
  although the roots $\pm t^{\onehalf}$ are conjugated in $\Z[t^{\pm\onehalf}]$ 
  over $\Z[t^\pm]$, this no longer holds in $\Z[A^\pm]$.
  Choosing this convention I have tried to reconcile simplicity and tradition, 
  so that all formulae become as simple as possible yet remain easily comparable.
  The results stated in the introduction are invariant
  under all possible normalizations and parametrizations, 
  but of course such conventions are important in actual 
  calculations and concrete examples.
\end{notation}

\subsection{Proof strategy} \label{sub:ProofStrategy}

\newcommand{\pica}[1]{\Bigl\langle\raisebox{-2ex}{\includegraphics[height=5ex,angle=0]{#1}}\Bigr\rangle}

Applying Kauffman's skein relation to a 
ribbon singularity, we obtain the following $16$ terms:
\begin{equation} \label{eq:BandSingularityResolution}
  \begin{split}
    \pica{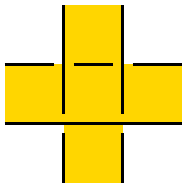} = & 
    + \pica{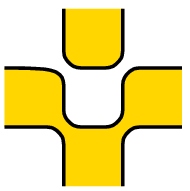} + \pica{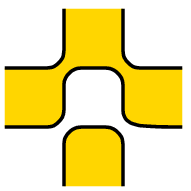}
    + \pica{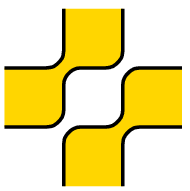} + \pica{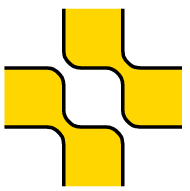} 
    \\ & 
    + \pica{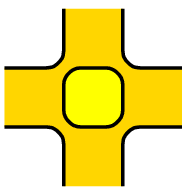} + \pica{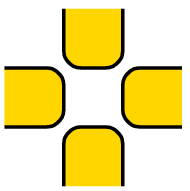} 
    + A^{+4} \pica{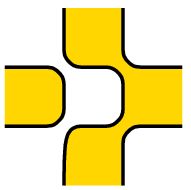} + A^{-4} \pica{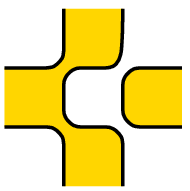}
    \\ & 
    + A^{+2} \pica{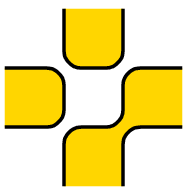} + A^{+2} \pica{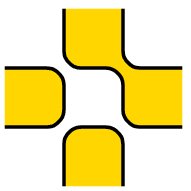} 
    + A^{-2} \pica{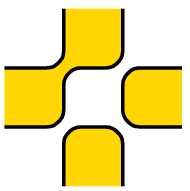} + A^{-2} \pica{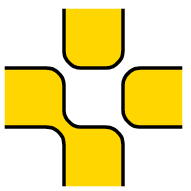}
    \\ & 
    + A^{+2} \pica{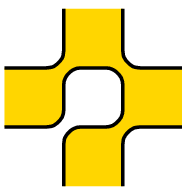} + A^{+2} \pica{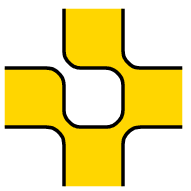} 
    + A^{-2} \pica{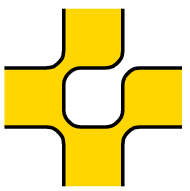} + A^{-2} \pica{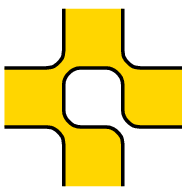} .
  \end{split}
\end{equation}

For a band crossing we obtain the same $16$ terms with permuted coefficients:
\begin{equation} \label{eq:BandCrossingResolution}
  \begin{split}
    \pica{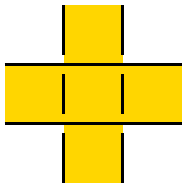} = & 
    + \pica{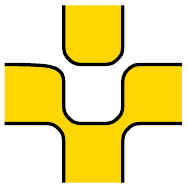} + \pica{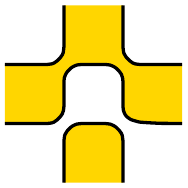} 
    + \pica{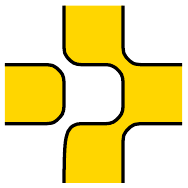} + \pica{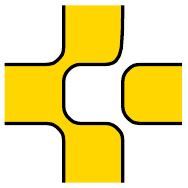} 
    \\ & 
    + \pica{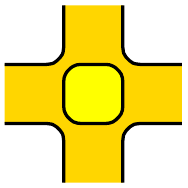} + \pica{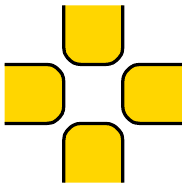} 
    + A^{+4} \pica{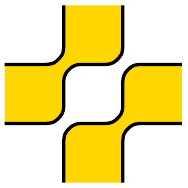} + A^{-4} \pica{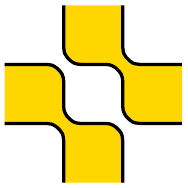} 
    \\ & 
    + A^{+2} \pica{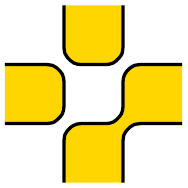} + A^{-2} \pica{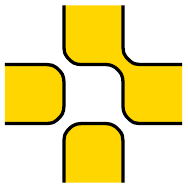} 
    + A^{+2} \pica{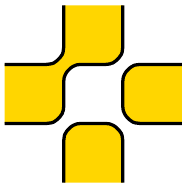} + A^{-2} \pica{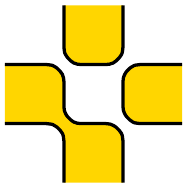} 
    \\ &
    + A^{+2} \pica{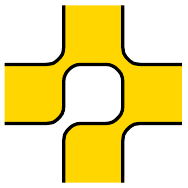} + A^{-2} \pica{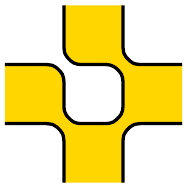} 
    + A^{+2} \pica{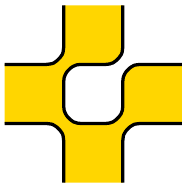} + A^{-2} \pica{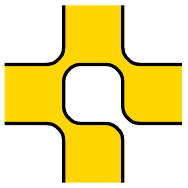} .
  \end{split}
\end{equation}

We are mainly interested in ribbon links, so we start out 
with a ribbon surface consisting only of disk components.  
Some of the resolutions displayed above, however, will lead 
to more complicated components, namely annuli or M\"obius bands.
We can avoid either M\"obius bands or annuli by adding half twists 
as desired, but we cannot avoid both of them altogether.
In order to set up an induction proof, this difficulty forces us 
to consider a suitable generalization including (at least) 
annuli or M\"obius bands. 

\begin{Notes}
We digress to mention that setting up
a ``loaded induction'' is a common phenomenon,
which has been so aptly formulated by George P\'olya:
\begin{quote}
  In general, in trying to devise a proof by mathematical 
  induction you may fail for two opposite reasons.  
  You may fail because you try to prove too much:  
  your [induction hypothesis] $P(n)$ is too heavy a burden.  
  Yet you may also fail because you try to prove too little: 
  your $P(n)$ is too weak a support.
  You have to balance the statement of your theorem 
  so that the support is just enough for the burden.
  \cite[p.\,119]{Polya:1954}
\end{quote}
For both Theorems \ref{Thm:RibbonJonesNullity} and 
\ref{Thm:MultiplicativityMod32} the principal difficulty 
is to find a suitable induction hypothesis.
It is hoped that P\'olya's aphorism will provide
some meta-mathematical explanation. 
\end{Notes}

\subsection{A lower bound for the Jones nullity} \label{sub:EulerJonesNullity}

Even though we are primarily interested in ribbon links, 
we are obliged to prove a more general statement, as motivated above.
The following seems to be the simplest setting supporting the desired inductive proof:

\begin{proposition} \label{prop:EulerJonesNullity}
  If a link $L \subset \R^3$ bounds an immersed ribbon surface 
  $S \subset \R^3$ of positive Euler characteristic $n$, 
  then $V(L)$ is divisible by $V(\bigcirc^n) = (q^+ + q^-)^{n-1}$.
\end{proposition}

Divisibility means that $V(L) = (q^+ + q^-)^{n-1} \, \tilde{V}$
for some $\tilde{V} \in \Z[q^\pm]$.  In this formulation the proposition
holds for all $n \in \Z$ but it is trivial, of course, for $n \le 1$.

\begin{example} \label{exm:HopfCable}
  The surface $S$ of \fullref{fig:BandSurfaceExample}a 
  has Euler characteristic $\chi(S) = 1+1+0 = 2$,
  so for the link $L = \partial S$ we expect $\Null V(L) \ge 1$.
  Indeed we find $\Null V(L) = 1$, because 
  \[ 
  V(L) = ( q^+ + q^- ) \cdot \bigl( q^6 - q^4 + 2q^2 + 2q^{-2} - q^{-4} + q^{-6} \bigr) .
  \]
  We also remark that $L = 8n8$ is the (anti-parallel) 
  $2$-cable of the Hopf link with zero framing.  
  It thus bounds a surface consisting of two annuli.
  According to the proposition, $L$ does not bound 
  a surface $S$ with $\chi(S)=3$ or $\chi(S)=4$.
\end{example}

\begin{figure}[h]
  \centering
  \subfigure[a surface with $\chi = 2$]{\qquad\includegraphics[height=14ex]{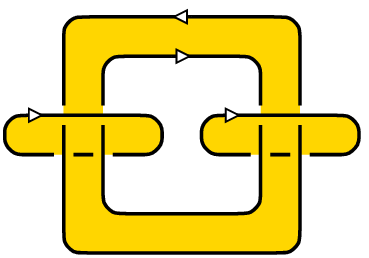}\qquad}
  \qquad
  \subfigure[a surface with $\chi = 1$]{\includegraphics[height=14ex]{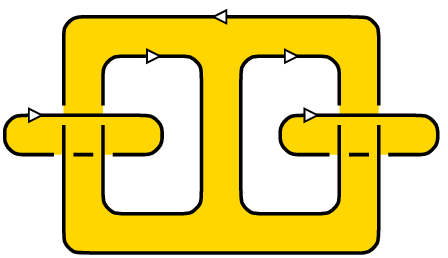}}
  \caption{Examples of ribbon surfaces}
  \label{fig:BandSurfaceExample}
\end{figure}

\begin{example} \label{exm:HopfSum}
  The surface $S$ of \fullref{fig:BandSurfaceExample}b
  has Euler characteristic $\chi(S) = 1+1-1 = 1$.
  The link $L = \partial S$ is the connected sum 
  $H_+ \cs H_- \cs H_+ \cs H_-$ of four Hopf links, whence
  \[
  V(L) = (q^{+1}+q^{+5})^2 \cdot (q^{-1}+q^{-5})^2 .
  \]
  We thus find $\det(L) = 16$ and $\Null V(L) = 0$.
  This example shows that the lower bound for $\Null V(L)$ 
  does not only depend on the number of disk components of $S$.
\end{example}

\begin{proof}[Proof of Proposition \ref{prop:EulerJonesNullity}]
  We proceed by induction on the \emph{ribbon number} $r(S)$ 
  of the ribbon surface $S$, that is, the number of ribbon singularities. 
  If $r(S) = 0$, then $S = S_0 \sqcup \bigcirc^n$, 
  and so $V(L)$ is divisible by $V(\bigcirc^n)$.
  To see this, notice that a connected surfaces with positive
  Euler characteristic is either a sphere, a projective plane, 
  or a disk.  Since $S$ has no closed components, this implies
  that $\chi(S) = n > 0$ can only be realized by (at least) $n$ disks.
  If the immersed surface $S$ has no singularities, then it is in fact 
  embedded in $\R^3$ and so $L$ has (at least) $n$ trivial components.

  For the induction step we assume that $r(S) \ge 1$ and that 
  the assertion is true for all ribbon surfaces $S'$ with $r(S') < r(S)$.  
  We replace one ribbon singularity by a band crossing, that is,
  \newcommand{\pic}[1]{\raisebox{-2ex}{\includegraphics[height=5ex]{#1}}}
  \begin{equation} \label{eq:Desingularization}
    \text{we transform }\quad
    S = \pic{ribbon-singularity}
    \quad\text{ into }\quad
    S' = \pic{ribbon-crossing} .
  \end{equation}
  For $L' = \partial S'$ we know by induction 
  that $V(L')$ is divisible by $V(\bigcirc^n)$.

  We represent the link $L$ and its ribbon surface $S$ 
  by a band diagram $D$.  Since the Jones polynomial $V(L)$ 
  and the bracket polynomial $\bracket{D}$ satisfy $V(L) = \pm A^k \bracket{D}$ 
  for some exponent $k \in \Z$, the assertion for the Jones polynomial 
  $V(L)$ and the bracket polynomial $\bracket{D}$ are equivalent.
  In the rest of the proof we will work with the latter.
  %
  %
  Subtracting Equations \eqref{eq:BandSingularityResolution}
  and \eqref{eq:BandCrossingResolution} we obtain the following difference:
  \begin{multline} \label{eq:BandSingularityRemoval}
    \pica{ribbon-singularity} - \pica{ribbon-crossing} = 
    (A^{+2}-A^{-2})\left( \pica{ribbon-resolve-BBAB} - \pica{ribbon-resolve-AAAB} \right)
    \\
    + (A^{+4}-1) \left( \pica{ribbon-resolve-BBAA} - \pica{ribbon-resolve-AAAA} \right)
    + (A^{-4}-1) \left( \pica{ribbon-resolve-AABB} - \pica{ribbon-resolve-BBBB} \right) .
  \end{multline}
  
  All ribbon surfaces on the right hand side 
  have ribbon number smaller than $r(S)$, so we can apply our induction hypothesis.
  Cutting open a band increases the Euler characteristic by one,
  whereas regluing decreases Euler characteristic by one:
  \[
  \newcommand{\euler}[1]{\chi\left(\raisebox{-1ex}{\includegraphics[height=3.2ex]{#1}}\right)}
  \euler{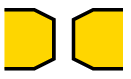} = \euler{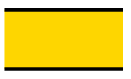} + 1 .
  \]
  On the right hand side of Equation \eqref{eq:BandSingularityRemoval},
  the two surfaces in the first parenthesis have Euler characteristic $n+1$,
  the four surfaces in the second parenthesis have Euler characteristic $n$, 
  so each bracket polynomial on the right hand side is divisible 
  by $\bracket{\bigcirc^n}$.  We conclude that $\bracket{D}$ is divisible 
  by $\bracket{\bigcirc^n}$, which completes the induction.
\end{proof}


\section{Band crossing changes} \label{sec:BandCrossingChanges}

Proposition \ref{prop:EulerJonesNullity} says that $V(L)$ 
is divisible by $V(\bigcirc^n) = (q^+ + q^-)^{n-1}$ whenever $L$ 
bounds a ribbon surface $S$ of positive Euler characteristic $n$.
The value $[L/S] := [V(L) / (q^+ + q^-)^{n-1}]_{(q \mapsto i)}$ 
is thus well-defined (\sref{sub:SurfaceDeterminant}),
and we show that it is invariant under certain operations
on the surface $S$, namely band crossing changes 
(\sref{sub:BandCrossingChanges}) and band twists (\sref{sub:BandTwists}).
We generalize these observations and establish a convenient framework 
by introducing the notion of surface invariants of finite type with respect 
to band crossing changes (\sref{sub:FiniteTypeSurfaceInvariants}).

\subsection{The surface determinant} \label{sub:SurfaceDeterminant}

We fix the following notation:

\begin{definition}[surface determinant]
  Consider an oriented link  $L \subset \R^3$ bounding an immersed 
  ribbon surface $S \subset \R^3$ with Euler characteristic $n=\chi(S)$.
  We define the \emph{determinant} of the surface $(S,L)$ 
  to be $[L/S] := [V(L) / (q^+ + q^-)^{n-1}]_{q \mapsto i}$.
\end{definition}

\begin{remark} \label{rem:SurfaceDetermninant}
  For $n=1$ this is the ordinary determinant, 
  $\det(L) = V(L)_{q \mapsto i}$.  For $n \le 0$ we multiply
  by $(q^+ + q^-)^{1-n}$ and evaluation at $q=i$ thus yields $[L/S] = 0$.

  We do not assume $S$ to be oriented, nor even orientable.
  In order to speak of the Jones polynomial $V(L)$, however, 
  we have to choose an orientation for the link $L = \partial S$.

  In general $[L/S]$ is not determined by the surface $S$ alone:
  although the unoriented link $L$ is determined by $S$, 
  the orientation of $L$ adds an extra bit of information.
  
  Likewise, $[L/S]$ is not an invariant of the link $L$ alone: 
  different surfaces may have different Euler characteristics,
  and we do not require $\chi(S)$ to be maximal.

  According to Proposition \ref{prop:EulerJonesNullity}
  the surface determinant $[L/S]$ is non zero only if 
  $\chi(S)\ge1$ and $S$ maximizes the Euler characteristic 
  of surfaces spanning $L$. 
\end{remark}

\begin{remark} \label{rem:Orientations}
  Changing the orientation of any link component changes the writhe 
  by a multiple of $4$ and thus $[L/S]$ changes by a factor $\pm1$.

  If we choose a surface component of $S$ and reverse
  the orientation of its entire boundary, then the writhe changes
  by a multiple of $8$, and so $[L/S]$ remains unchanged. 

  This applies in particular to reversing a link component 
  that bounds a disk or M\"obius band.  If there are no other 
  components, then $[L/S]$ is independent of orientations.
\end{remark}  

\begin{example} \label{exm:SignOfDeterminant}
  We always have $[L/S] \in \Z$ or $[L/S] \in i\Z$,
  depending on whether $c(L) - \chi(S)$ is even or odd.
  Here are two simple examples:
  \newcommand{\pic}[1]{\raisebox{-4.5ex}{\includegraphics[height=10ex]{#1}}}
  \begin{equation} \label{eq:EssentialExample}
    S = \pic{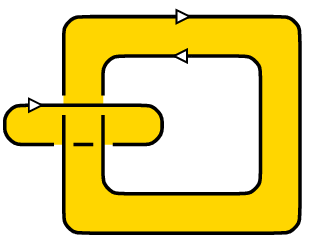}
    \quad\text{ and }\quad
    S' = \pic{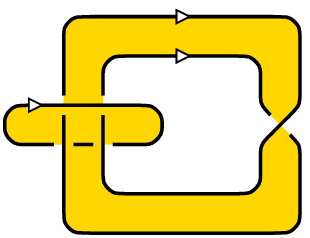} .
  \end{equation}
  We have $[L/S] = \det L = 4$ and $[L'/S'] = \det L' = -4i$.  
  We remark that $L = H_- \cs H_+$, and we can change orientations 
  so as to obtain $H_- \cs H_-$ or $H_+ \cs H_+$, both with determinant $-4$.
  Notice that $[L'/S']$ is independent of orientations.
\end{example}

\subsection{Band crossing changes} \label{sub:BandCrossingChanges}

The following observation will be useful:

\begin{proposition} \label{prop:BandCrossingChange}
  The surface determinant $[L/S]$ is invariant 
  under band crossing changes.
\end{proposition}

\begin{proof}
  Let $S$ be an immersed ribbon surface of positive Euler characteristic $n = \chi(S)$.
  We reconsider Equation \eqref{eq:BandCrossingResolution},
  resolving a band crossing according to Kauffman's skein relation.
  The resolution for the changed band crossing is analogous, 
  with all diagrams rotated by $90^\circ$.  When we calculate 
  their difference, $10$ of the $16$ terms cancel each other
  in pairs, and we obtain the following skein relation:
  \begin{multline} \label{eq:BandCrossingChange}
    \pica{ribbon-crossing} - \pica{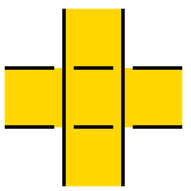} = (A^{+4}-A^{-4}) 
    \left( \pica{ribbon-resolve-AAAA} - \pica{ribbon-resolve-BBBB} \right) \\
    + (A^{+2}-A^{-2})\left( \pica{ribbon-resolve-BAAA} - \pica{ribbon-resolve-BBAB} 
      + \pica{ribbon-resolve-AAAB} - \pica{ribbon-resolve-BABB} \right) .
  \end{multline}
  The two surfaces in the first parenthesis have Euler characteristic $n$,
  so their polynomials are divisible by $\bracket{\bigcirc^n} = (q^+ + q^-)^{n-1}$, 
  and the coefficient $q^{-2} - q^{+2} = (q^- + q^+) (q^- - q^+)$ contributes another factor.
  The four surfaces in the second parenthesis have Euler characteristic $n+1$, 
  so their polynomials are divisible by $\bracket{\bigcirc^{n+1}} = (q^+ + q^-)^n$.

  This means that $\bracket{D}$ modulo $(q^+ + q^-)^n$ 
  is invariant under band crossing changes as stated.  
  The writhe remains constant or changes by $\pm 8$.
  We conclude that the Jones polynomial $V(L)$ modulo $(q^+ + q^-)^n$ 
  is invariant under band crossing changes.  
\end{proof}

\subsection{Band twists} \label{sub:BandTwists}

Generalizing Example \ref{exm:SignOfDeterminant}, 
we obtain the following result:

\begin{proposition} \label{prop:TwistInvariance}
  \newcommand{\pic}[1]{\bigl[\raisebox{-1ex}{\includegraphics[height=3.2ex]{#1}}\bigr]}
  The surface determinant $[L/S]$ is invariant under band twisting,
  up to some sign factor $\varepsilon \in \{\pm 1, \pm i\}$.  More precisely, we have 
  \begin{align*}
    \pic{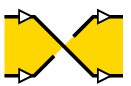} &  = i \pic{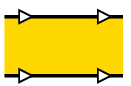} = - \pic{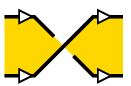}
    \qquad \text{and} \\
    \pic{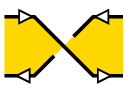} & = \pic{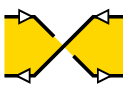} = \varepsilon \pic{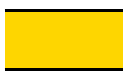} .
  \end{align*}
\end{proposition}

For the last link we have to choose arbitrary orientations;
there is no canonical choice.

\begin{proof}
  \newcommand{\pic}[1]{\raisebox{-1ex}{\includegraphics[height=3.2ex]{#1}}}
  We have $\bracket{\pic{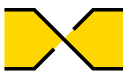}} 
  = A \bracket{\pic{ribbon-straight}} 
  + A^{-1} \bracket{\pic{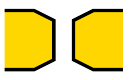}}$.
  The last term does not contribute to the surface determinant
  because it has greater Euler characteristic.
  The other two terms establish the desired equality 
  upon normalization with respect to the writhe.
  For parallel orientations we obtain:
  \begin{align*}
    \bigl[ \pic{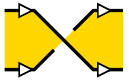} \bigr] 
    & = \bigl[ (-A^{-3})^{w+1} \bracket{\pic{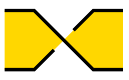}}/(q^+ + q^-)^{\chi-1} \bigr]_{(q \mapsto i)}
    \\
    & = \bigl[ -A^{-2} \bigr]_{(q \mapsto i)} \cdot
    \bigl[ (-A^{-3})^w \bracket{\pic{ribbon-straight}}/(q^+ + q^-)^{\chi-1} \bigr]_{(q \mapsto i)} 
    = i \bigl[ \pic{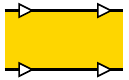} \bigr] .
  \end{align*}
  We recall our sign convention $q = -A^{-2}$.
  For anti-parallel orientations we obtain:
  \begin{align*}
    \bigl[ \pic{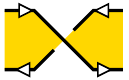} \bigr] 
    & = \bigl[ (-A^{-3})^{w-1} \bracket{\pic{ribbon-twist-1}}/(q^+ + q^-)^{\chi-1} \bigr]_{(q \mapsto i)}
    \\
    & = \bigl[ -A^{4} \bigr]_{(q \mapsto i)} \cdot
    \bigl[ (-A^{-3})^w \bracket{\pic{ribbon-straight}}/(q^+ + q^-)^{\chi-1} \bigr]_{(q \mapsto i)} 
    \\
    & = \bigl[ A^{8} \bigr]_{(q \mapsto i)} \cdot
    \bigl[ (-A^{-3})^{w+1} \bracket{\pic{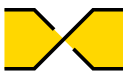}}/(q^+ + q^-)^{\chi-1} \bigr]_{(q \mapsto i)} 
    = \bigl[ \pic{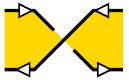} \bigr] .
  \end{align*}
  The middle term can be identified with $\pm \bigl[ \pic{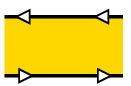} \bigr]$
  or $\pm i \bigl[ \pic{ribbon-straight-2} \bigr]$.  In general the sign depends on 
  the chosen orientations and the induced writhe of the resulting diagram.
\end{proof}

\begin{remark} \label{rem:DoubleTwist}
  Adding two full band twists leaves $[L/S]$ invariant.
  Of course, this can also be realized by a band crossing change,
  as in the following example: 
  \[
  \includegraphics[height=10ex]{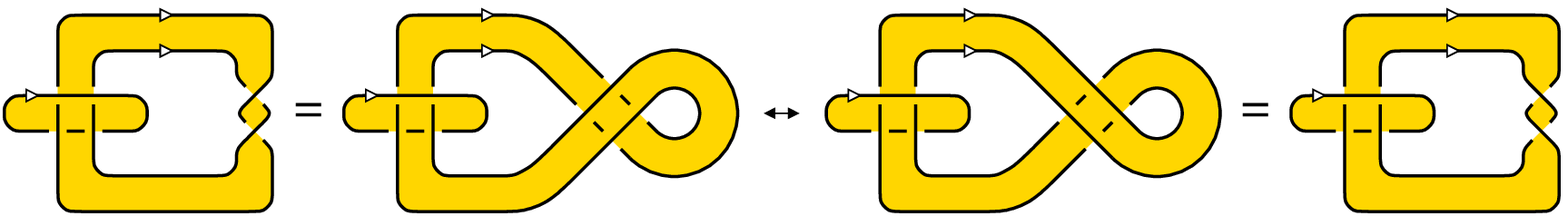}
  \]
\end{remark}

\begin{example} \label{exm:HopfLink}
  As a word of warning, the following example illustrates 
  that band twisting or a band crossing change 
  can alter the determinant of a link.  We consider 
  \[
  \newcommand{\pic}[1]{\raisebox{-3ex}{\includegraphics[height=8ex]{#1}}}
  S_- = \pic{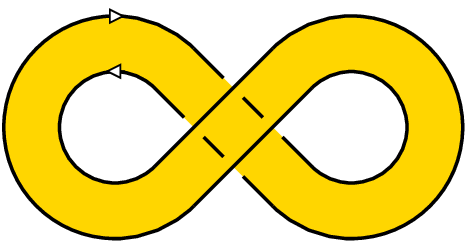} \qquad\text{and}\qquad S_+ = \pic{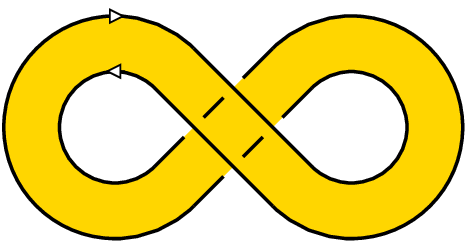} .
  \]
  The boundary $H_\pm = \partial S_\pm$ is the Hopf link with 
  linking number $\lk(H_\pm) = \pm 1$ and determinant $\det H_\pm = \pm 2i$.
  Propositions \ref{prop:BandCrossingChange} and \ref{prop:TwistInvariance} 
  apply to the surface determinant $[L_\pm/S_\pm]$, but 
  the statement is empty because $[L_\pm/S_\pm] = 0$.
  See Remark \ref{rem:SurfaceDetermninant}.
\end{example}

\subsection{Orientable surfaces} \label{sub:OrientableSurfaces}

In order to simplify the exposition we will concentrate 
on \emph{orientable} surfaces.  This restriction seems
acceptable because we are ultimately interested in ribbon links.
All results extend to non-orientable surfaces as well, but statements 
and proofs are twice as long due to clumsy case distinctions.

\begin{definition}
  If $S$ is orientable then we define $[S] := [L/S]$ by choosing an arbitrary
  orientation of $S$ and the induced orientation of the boundary $L = \partial S$.
  This is well-defined according to Remark \ref{rem:Orientations}.
\end{definition}

\begin{proposition}
  For every orientable surface $S$ we have $[S] \in \Z$.
\end{proposition}

\begin{proof}
  We have $\chi(S) \equiv c(L) \mod{2}$, where $c(L)$ 
  is the number of components of the link $L = \partial S$.
  The Jones polynomial $V(L)$ is even if $c(L)$ is odd,
  and $V(L)$ is odd if $c(L)$ is even.  The reduced polynomial 
  $V(L) / (q^+ + q^-)^{\chi-1} \in \Z[q^\pm]$ is always even.
\end{proof}

\begin{remark}
  By definition, $[S]$ depends only on the link $L = \partial S$ 
  and the Euler characteristic $\chi(S)$ of the surface $S$.
  According to Propositions \ref{prop:BandCrossingChange}
  and \ref{prop:TwistInvariance}, the value $[S]$ 
  does not depend on the situation of $S$ in $\R^3$, 
  but only on the abstract surface together with 
  the combinatorial pattern of ribbon singularities.
  This is rather surprising.
\end{remark}

\subsection{Surface invariants of finite type} \label{sub:FiniteTypeSurfaceInvariants}

In order to put the surface determinant into a wider perspective,
I would like to expound an interesting analogy with link invariants of finite type.  
A more comprehensive study of surface invariants of finite type 
will be the object of a forthcoming article \cite{Eisermann:SurfInv}.

\begin{remark}
  We expand the Jones polynomial $V(q) = \sum_{k=0}^\infty v_k h^k$ in $q = \exp(h/2)$.  
  Here any power series $q \equiv 1 + h/2 \mod{h^2}$ could be used: 
  the crucial point is that $q - q^{-1} \equiv h$ has no constant term.
  Then the link invariants $L \mapsto v_k(L)$ are of finite type in the sense 
  of Vassiliev \cite{Vassiliev:1990} and Goussarov \cite{Goussarov:1991}, 
  see also Birman--Lin \cite{BirmanLin:1993} and Bar-Natan \cite{BarNatan:1995}.  
  This means that these invariants behave polynomially with respect to crossing changes 
  $\raisebox{-0.7ex}{\includegraphics[height=2.5ex]{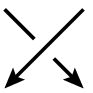}} \leftrightarrow 
  \raisebox{-0.7ex}{\includegraphics[height=2.5ex]{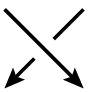}}$.
\end{remark}

\begin{remark}
  We can also expand $V(q) = \sum_{k=0}^\infty d_k h^k$ in $q = i \exp(h/2)$.  
  Any power series $q \equiv i + i h/2 \mod{h^2}$ could be used: 
  the crucial point is that $q + q^{-1} \equiv i h$ has no constant term.
  We obtain a family of link invariants $L \mapsto d_k(L)$ starting with 
  the classical link determinant $d_0(L) = V(L)|_{(q \mapsto i)} = \det(L)$. 
  The Jones nullity $\nu = \Null V(L)$ is the smallest index such that $d_\nu(L) \ne 0$.
  If $L$ bounds a surface $S$ of positive Euler characteristic $n$, 
  then $d_0(L)=\dots=d_{n-2}(L)=0$ and $d_{n-1}(L) = i^{n-1} [L/S]$. 
\end{remark}

The arguments used in the proofs of Propositions 
\ref{prop:BandCrossingChange} and \ref{prop:TwistInvariance}
motivate the following definition of alternating sums
of surfaces, imitating finite type invariants of links.

\begin{notation}
  As in \sref{sub:BandDiagrams} we consider a smooth compact surface $\Sigma$ 
  without closed components.  In order to simplify we assume $\Sigma$ 
  to be oriented and endow $\partial \Sigma$ with the induced orientation.
  We denote by $\BImmersions(\Sigma)$ the set of band immersions 
  $\Sigma \looparrowright \R^3$ modulo ambient isotopy.

  Let $D$ be a band diagram representing some ribbon surface $S \in \BImmersions(\Sigma)$
  and let $X$ be a set of band crossings of $D$.  For each subset $Y \subset X$ 
  we denote by $D_Y$ the diagram obtained from $D$ by changing all 
  band crossings $x \in Y$ as indicated in \fullref{fig:CrossingChange}a.  
  This does not change the abstract surface $\Sigma$, and so the diagram $D_Y$ 
  represents again a ribbon surface in $\BImmersions(\Sigma)$.

  \begin{figure}[h]
    \centering
    \subfigure[Band crossing]{\begin{minipage}[c][9ex][c]{20ex}\centering
        \newcommand{\pic}[1]{\raisebox{-3.5ex}{\includegraphics[height=8ex]{#1}}}
        $\pic{ribbon-crossing} \leftrightarrow \pic{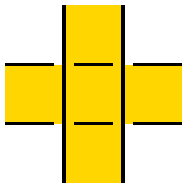}$
      \end{minipage}}
    \qquad
    \subfigure[Band twist]{\begin{minipage}[c][9ex][c]{20ex}\centering
        \newcommand{\pic}[1]{\raisebox{-1.5ex}{\includegraphics[height=4ex]{#1}}}
        $\pic{ribbon-twist-1} \leftrightarrow \pic{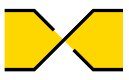}$
      \end{minipage}}
    \caption{Changing band crossings and band twists}
    \label{fig:CrossingChange}
  \end{figure}

  Slightly more generally, we also allow $X$ to contain band twists, 
  in which case we simply change one crossing as in \fullref{fig:CrossingChange}b. 
  We will usually not mention this explicitly but rather 
  subsume it under the notion of band crossing change.
  Of course, two full band twist can be traded for one 
  band crossing change, see Remark \ref{rem:DoubleTwist}.
\end{notation}

\begin{remark}
  We emphasize that we are considering links $L$ 
  equipped with extra structure, namely the given surface $S \subset \R^3$ 
  spanning $L = \partial S$.  This extra structure is crucial.
  Kauffman \cite[chapter V]{Kauffman:OnKnots} studied \emph{pass moves}, 
  which consist of the move of \fullref{fig:CrossingChange}a 
  without keeping track of surfaces.
  He shows that the set of knots splits into two equivalence classes,
  corresponding to the two values of the Arf invariant.
\end{remark}

\begin{remark}
  We assume that $\Sigma$ is a compact surface without closed components.
  Then any two embeddings $f,g \co \Sigma \into \R^3$ can be transformed 
  one into the other by a finite sequence of the above band crossing changes.
  The same holds true for ribbon immersions $f,g \co \Sigma \looparrowright \R^3$
  provided that the combinatorial structure of their singularities co\"incide.
\end{remark}

\begin{definition}
  Let $v \co \BImmersions(\Sigma) \to A$ 
  be a surface invariant with values in some abelian group $A$. 
  We say that $v$ is of \emph{degree $\le m$} 
  with respect to band crossing changes if
  \[
  \sum_{Y \subset X} (-1)^{|Y|} \, v(D_Y) = 0 
  \qquad\text{whenever}\qquad |X| > m .
  \]
  We say that $v$ is a \emph{surface invariant of finite type}
  if it is of degree $\le m$ for some $m \in \N$.
\end{definition}

\begin{remark}
  If $A$ is a module over a ring $\K$, 
  then the surface invariants $\BImmersions(\Sigma) \to A$
  of degree $\le m$ form a module over $\K$.
  If $A$ is an algebra over $\K$, then 
  the surface invariants $\BImmersions(\Sigma) \to A$
  of finite type form a filtered algebra over $\K$:
  if $f$ is of degree $\le m$ and $g$ is of degree $\le n$, 
  then their product $f \cdot g$ is of degree $\le m + n$.
\end{remark}

\begin{proposition} \label{prop:FiniteTypeSurfaceInvariants}
  The surface invariant $S \mapsto d_k(\partial S)$
  is of finite type with respect to band crossing changes.
  More precisely, it is of degree $\le m$ for $m = k + 1 - \chi(S)$.
\end{proposition}

In the case $m<0$ we have $k < \chi(S)-1$, whence 
$d_k(\partial S) = 0$ by Proposition \ref{prop:EulerJonesNullity}.
For $m=0$ we have $k = \chi(S)-1$, whence $d_k(\partial S)$ 
is invariant under band crossing changes
by Proposition \ref{prop:BandCrossingChange}.
Proposition \ref{prop:FiniteTypeSurfaceInvariants}
extends these results in a natural way to all $k \in \N$.
It is a consequence of the following observation:

\begin{lemma} \label{lem:FiniteTypeSurfaceInvariants}
  Consider an oriented band diagram $D$. 
  For every set $X$ of band crossings, the polynomial
  $\sum_{Y \subset X} (-1)^{|Y|} \, V(\partial D_Y)$
  is divisible by $(q^+ + q^-)^{|X| + \chi(S) - 1}$.
\end{lemma}

\begin{proof}
  We proceed by induction on the cardinality of $X$.
  The case $|X|=0$ is settled by Proposition \ref{prop:EulerJonesNullity}.
  If $|X| \ge 1$ then we choose one band crossing or band twist $x \in X$.
  In the first case we apply Equation \eqref{eq:BandCrossingChange}.
  The orientations of the vertical and horizontal strands 
  are antiparallel, so we can put them into the following configuration:
  \newcommand{\skein}[1]{V\Bigl(\raisebox{-2ex}{\includegraphics[height=5ex]{#1}}\Bigr)}
  \begin{multline} \label{eq:FiniteTypeSurfaceInvariants}
    \skein{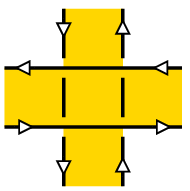} - \skein{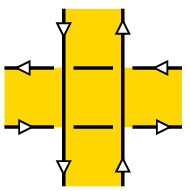} = (q^- - q^+)(q^+ + q^-)
    \left[ \skein{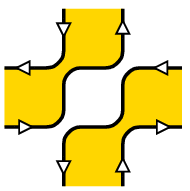} - \skein{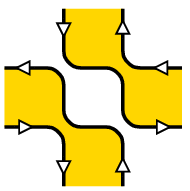} \right] \\
    + (q^+ - q^-)\left[ \skein{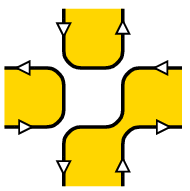} - \skein{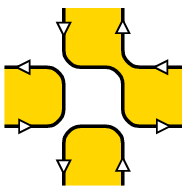} 
      + \skein{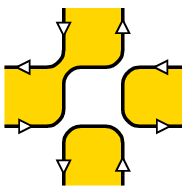} - \skein{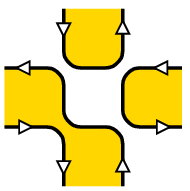} \right] .
  \end{multline}
  The diagrams so obtained have the same writhe, and thus 
  Equation \eqref{eq:BandSingularityRemoval} for the Kauffman bracket 
  directly translates to Equation \eqref{eq:FiniteTypeSurfaceInvariants}
  for the Jones polynomial.  On the right hand side the first two terms 
  have the same Euler characteristic as $S$ but one extra factor $(q^++q^-)$, 
  whereas in the last four terms the Euler characteristic increases by one.

  The second case is analogous: if $x$ is a band twist, then Equation \eqref{eq:HOMFLYPT} yields
  \renewcommand{\skein}[1]{V\bigl(\raisebox{-1ex}{\includegraphics[height=3ex]{#1}}\bigr)}
  \begin{align}
    \skein{ribbon-otwist-2} - \skein{ribbon-otwist-1}
    & =
    \left[ \skein{ribbon-otwist-2} - q^4 \skein{ribbon-otwist-1} \right]
    + ( q^4-1 ) \skein{ribbon-otwist-1} 
    \\ \notag
    & = (q^1-q^3) \skein{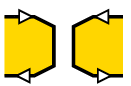}
    + (q^3-q^1)(q^++q^-) \skein{ribbon-otwist-1} 
  \end{align}

  In both cases we pass to the alternating sum over all 
  subsets $Y$ of $X' = X \minus \{x\}$.  On the left hand side 
  we obtain the alternating sum over all subsets of $X$, as desired.  
  On the right hand side we apply the induction hypothesis
  to conclude that the resulting polynomial is divisible 
  by $(q^+ + q^-)^{|X| + \chi(S) - 1}$.
\end{proof}

\begin{remark}
  Every link invariant $L \mapsto v(L)$ of degree $\le m$ 
  (with respect to crossing changes) induces a surface invariant 
  $S \mapsto v(\partial S)$ of degree $\le m$
  (with respect to band crossing changes).  
  This holds, for example, for the coefficients $v_k$ in 
  the above expansion $V(q) = \sum_{k=0}^\infty v_k h^k$ in $q = \exp(h/2)$.
  It is surprising that the expansion $V(q) = \sum_{k=0}^\infty d_k h^k$ 
  in $q = i \exp(h/2)$ provides an independent family of examples, 
  even though the link invariants $d_k$ are \emph{not} 
  of finite type with respect to crossing changes.  
  
  The determinant $d_0(L) = \det(L)$ comes close to being 
  a Vassiliev--Goussarov invariant in the sense that $\det(L)^2$ 
  is polynomial of degree $\le 2$ on every twist sequence,
  see Eisermann \cite[\textsection 5]{Eisermann:2003}.
  Here $\det(\partial S)$ turns out to be of degree $\le 1-\chi(S)$
  with respect to band crossing changes of the surface $S$.
\end{remark}

\begin{remark} \label{rem:AlexanderBandCrossing}
  We parametrize the Alexander--Conway polynomial 
  $\Delta(L) = \sum a_k(L) z^k$ by $z = q^+ - q^-$.
  The link invariant $L \mapsto a_k(L)$ is then 
  of degree $k$ with respect to crossing changes.  
  If we consider a disk $\Sigma$ and band immersions
  $\Sigma \looparrowright \R^3$, then the surface invariant
  $S \mapsto a_k(\partial S)$ is of degree $0$
  with respect to band crossing changes.
  To see this, notice that the Seifert matrix of the knot $K = \partial S$ 
  has the form $\theta = \bigl(\begin{smallmatrix} 0 & A \\ B & C \end{smallmatrix}\bigr)$,
  see Kauffman \cite[chapter VIII]{Kauffman:OnKnots}. 
  This implies that $\Sign(K) = \Sign(\theta+\theta^*)$ vanishes and 
  that $\Delta(K) = \det( q^- \, \theta^* - q^+ \, \theta )$ 
  is of the form $f(q^+) \cdot f(q^-)$ with $f \in \Z[q^\pm]$. 
  Band crossing changes of $S$ only affect the submatrix $C$,
  and so $\Delta(K)$ remains unchanged.

  If we pass from the special case of a disk 
  to immersions or embeddings of an arbitrary surface $\Sigma$,
  then the surface invariant $S \mapsto a_k(\partial S)$ is no longer 
  invariant under band crossing changes.  Example \ref{exm:HopfLink} illustrates
  this for the linking number $a_1 = \lk$ when $\Sigma$ is an annulus.
\end{remark}


\section{The Jones determinant of ribbon links} \label{sec:JonesDeterminant}

The surface determinant $[S]$ is invariant under 
band crossing changes, but in general it changes when 
we replace a ribbon singularity by a band crossing.
In order to analyze this in more detail, we spell out
an oriented skein relation (\sref{sub:OrientedSkein}) and
establish some useful congruences (\sref{sub:Congruences}).
We then apply them to ribbon links (\sref{sub:RibbonLinks}) and 
prove \fullref{Thm:MultiplicativityMod32} stated in the introduction.
(The arguments remain elementary but get increasingly complicated,
because our combinatorial approach entails numerous case distinctions.)
Finally we sketch an application to satellites of ribbon knots
(\sref{sub:RibbonSatellites}).

\subsection{An oriented skein relation} \label{sub:OrientedSkein}

We wish to set up a suitable skein relation 
for the determinant $[S]$ of an orientable ribbon surface $S$.
Replacing a ribbon singularity by a band crossing 
as in Equation \eqref{eq:BandSingularityRemoval},
we obtain a ribbon surface $S'$ with one less singularity.
The right hand side of \eqref{eq:BandSingularityRemoval} 
features six diagrams: the first two of these terms 
vanish at $q=i$ because they have greater Euler characteristic.
Hence Equation \eqref{eq:BandSingularityRemoval} becomes
\begin{equation} \label{eq:Determinant}
  \newcommand{\pic}[2]{\biggl[\underset{#2}{%
      \raisebox{-2.5ex}{\includegraphics[height=6ex]{#1}}}\biggr]}
  \pic{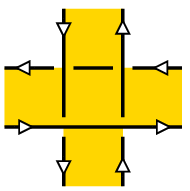}{S} - \pic{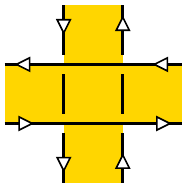}{S'} = 
  -2 \biggl( \pic{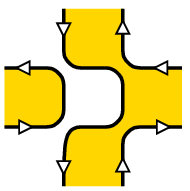}{S_1} - \pic{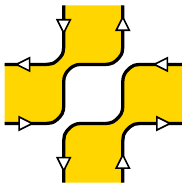}{S_2} 
  + \pic{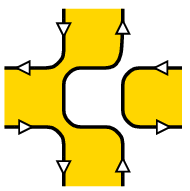}{S_3} - \pic{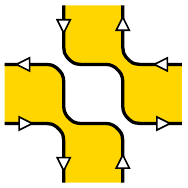}{S_4} \biggr) .
\end{equation}

Notice that the orientations of the vertical and horizontal strands
are antiparallel, and the writhe of the shown crossings in $S$ and $S'$ 
add up to $0$.  Inserting pairs of opposite twists as necessary, 
we can always put the bands into the configuration shown in \eqref{eq:Determinant}.
This has the advantage that we can use the same orientations on the right hand side.
All diagrams have the same writhe, so that Equation \eqref{eq:BandSingularityRemoval} 
for the Kauffman bracket directly translates to the Jones polynomial, and to 
Equation \eqref{eq:Determinant} for the surface determinant.

\subsection{Some useful congruences} \label{sub:Congruences}
  
We continue to consider an orientable ribbon surface $S$.
We denote by $c(S)$ the number of its connected components.
Since each component has Euler characteristic $\le 1$,
the deficiency $d(S) = c(S) - \chi(S)$ is non-negative,
and we have $d(S) = 0$ if and only if $S$ consists only of disks.
In the following induction the deficiency $d(S)$ measures 
how far $S$ is from being a collection of disks.

\begin{definition}
  We call a ribbon singularity \emph{essential} if 
  the pierced component is a disk and the piercing component 
  remains connected after cutting it open along the singularity.
  We denote by $e(S)$ the number of essential singularities of $S$.
\end{definition}

\begin{lemma} \label{lem:Congruence}
  Every oriented ribbon surface $S \subset \R^3$
  satisfies the following congruences:
  \begin{enumerate}
  \item[(1)]
    If $d(S)=0$, then $[S] \equiv 1 \mod{8}$.
  \item[(2)]
    If $d(S)=1$, then $[S] \equiv 4 e(S) \mod{8}$.
  \item[(3)]
    If $d(S)\ge2$, then $[S] \equiv 0 \mod{2^{d+1}}$.
  \end{enumerate}
\end{lemma}

\begin{remark}
  The ribbon condition improves the usual congruences by a factor $2$: 
  in general we only have $\det K \equiv 1 \mod{4}$ 
  for a knot and $\det L \equiv 0 \mod{2}$ for a link.

  Case (1) includes the well-known fact that every 
  ribbon knot $K$ satisfies $\det(K) \equiv 1 \mod{8}$.
  This classical result is reproved in our more 
  general setting for ribbon links.

  Case (2) could be reduced to $[S] \equiv 0 \mod{4}$,
  but the refinement modulo $8$ will prove indispensable 
  in order to establish \fullref{Thm:MultiplicativityMod32}
  (see Theorem \ref{thm:CongruenceMod32} below).

  Case (3) could likewise be strengthened,
  but we content ourselves with a weaker formulation 
  that suffices for the inductive proof of Lemma \ref{lem:Congruence}.
\end{remark}

\begin{proof}[Proof of Lemma \ref{lem:Congruence}]
  We first recall that we assume the surface $S$ 
  to be non-empty and without closed components.
  We also remark that the case $\chi(S) \le 0$ is trivial,
  because $d(S) \ge 1$ and $[S]=0$ by definition.
  In the sequel we can thus assume $\chi(S) \ge 1$.

  We proceed by induction on the number $r(S)$ of ribbon singularities.
  Suppose first that $r(S) = 0$.  
  If $d(S) = 0$, then $S = \bigcirc^\chi$ with $\chi = \chi(S)$, whence $[S] = 1$.  
  If $d(S) \ge 1$ then $S = S_0 \sqcup \bigcirc^\chi$ with $S_0 \ne \emptyset$, 
  whence $V(L) = (q^+ + q^-)^\chi \, V(L_0)$ and $[S] = 0$.  

  For the induction step we suppose that $r(S) \ge 1$
  and that the statement is true for all surfaces $S'$ with $r(S') < r(S)$.
  We then replace a ribbon singularity of $S$ by a band crossing 
  as in Equation \eqref{eq:Determinant}.
  By our induction hypothesis, we can apply the congruences
  stated above to the surface $S',S_1,S_2,S_3,S_4$. 
  All surfaces have the same Euler characteristic as $S$
  but the number of components may differ: we have 
  $c(S) = c(S')$ and $c(S_i) - c(S) \in \{ 1,0,-1 \}$.

  We denote by $S_\Horz$ resp.\ $S_\Vert$ the component the surface $S$
  containing the horizontal resp.\ vertical strip in Equation \eqref{eq:Determinant}.
  In order to analyze the contribution of the four ribbon surfaces 
  $S_1,S_2,S_3,S_4$ we distinguish the following cases.
  
  \textbf{Case (1)} \enspace
  If $d(S)=0$, then we are dealing exclusively with disks:
  \begin{enumerate}
  \item[(a)]
    If $S_\Horz$ and $S_\Vert$ are different disks of $S$,
    then all four diagrams $S_1,S_2,S_3,S_4$ feature only disks, 
    whence $d(S_1) = d(S_2) = d(S_3) = d(S_4) = 0$.
    We thus have $[S_1] \equiv [S_2] \equiv [S_3] \equiv [S_4] \equiv 1 \mod{8}$,
    whence $[S_1] - [S_2] + [S_3] - [S_4] \equiv 0 \mod{8}$.
    The factor $-2$ in Equation \eqref{eq:Determinant}
    ensures that $[S] \equiv [S'] \mod{16}$.
    
  \item[(b)]
    Suppose next that $S_\Vert$ co\"incides with $S_\Horz$.
    For concreteness we will assume that the western and southern 
    pieces are connected outside of the local picture,
    as indicated in \fullref{fig:RibbonResolveSubtlety1}.
    (The other three variants are analogous.)

    \begin{figure}[h]
      \centering
      \includegraphics[width=\linewidth]{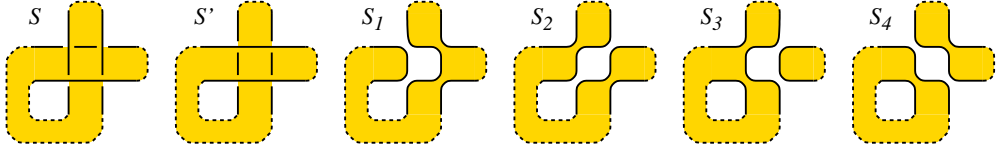}
      \caption{Resolving a pure ribbon singularity}
      \label{fig:RibbonResolveSubtlety1}
    \end{figure}
    
    Two diagrams, in our case $S_1$ and $S_2$, 
    feature only disks, whence $d(S_1) = d(S_2) = 0$
    and $[S_1] \equiv [S_2] \equiv 1 \mod{8}$.
    The other two diagrams, in our case $S_3$ and $S_4$,  
    each feature one extra annulus, whence $d(S_3) = d(S_4) = 1$,
    whence $[S_3] \equiv [S_4] \equiv 0 \mod{4}$.
    Equation \eqref{eq:Determinant} implies 
    that $[S] \equiv [S'] \mod{8}$.

    \begin{remark} \label{rem:DeltaMod16}
      For future reference, we wish to be more precise here. 
      The surfaces $S_3$ and $S_4$ may have different numbers 
      of essential singularities, so $[S_4] - [S_3] = 4 \Delta \mod{8}$.
      We conclude that $[S] - [S'] \equiv 8\Delta \mod{16}$:
      if $\Delta$ is even, then $[S] \equiv [S'] \mod{16}$;
      if $\Delta$ is odd, then $[S] - [S'] \equiv 8 \mod{16}$.
      
      The difference $\Delta = e(S_4) - e(S_3)$ is the number 
      of times that the annulus formed by connecting the western 
      and southern pieces, essentially pierces the northern piece:
      these singularities are essential for $S_4$ but not essential for $S_3$.
      All other essential singularities are the same for both $S_3$ and $S_4$.
    \end{remark}
    
  \end{enumerate}

  \textbf{Case (2)} \enspace
  If $d(S)=1$, then we are dealing with $n$ disks and one annulus.  
  
  \begin{enumerate}
  \item[(a)]
    If the components $S_\Horz$ and $S_\Vert$ co\"incide,
    then $d(S_1) = d(S_2) = d(S_3) = d(S_4) \ge 1$, whence 
    $[S_1] \equiv [S_2] \equiv [S_3] \equiv [S_4] \equiv 0 \mod{4}$.
    In this case $[S] \equiv [S'] \mod{8}$.  The considered
    singularity is not essential, so that $e(S) = e(S')$.
  \end{enumerate}
 
  In the following cases we assume that 
  $S_\Horz$ and $S_\Vert$ are different components.
  \begin{enumerate}
  \item[(b)]
    If splitting separates both $S_\Horz$ and $S_\Vert$, then 
    $d(S_1) = d(S_2) = d(S_3) = d(S_4) = 1$, whence 
    $[S_1] \equiv [S_2] \equiv [S_3] \equiv [S_4] \equiv 0 \mod{4}$.
    We conclude that $[S] \equiv [S'] \mod{8}$.  The considered
    singularity is not essential, so that $e(S) = e(S')$.

  \item[(c)]
    If splitting separates $S_\Vert$ but not $S_\Horz$, 
    then $d(S_1) = d(S_2) = d(S_3) = d(S_4) = 0$,
    whence $[S_1] \equiv [S_2] \equiv [S_3] \equiv [S_4] \equiv 1 \mod{8}$.
    We conclude that $[S] \equiv [S'] \mod{16}$.
    The considered singularity is not essential, so that $e(S) = e(S')$.

  \item[(d)]
    If splitting separates $S_\Horz$ but not $S_\Vert$, 
    then $d(S_1) = d(S_3) = 1$ and $d(S_2) = d(S_4) = 0$.
    We thus have $[S_1] \equiv [S_3] \equiv 0 \mod{4}$ 
    as well as $[S_2] \equiv [S_4] \equiv 1 \mod{8}$,
    whence $[S_1] - [S_2] + [S_3] - [S_4] \equiv 2 \mod{4}$.
    We conclude that $[S] \equiv [S'] + 4 \mod{8}$.
    The considered singularity is essential, 
    so that $e(S) = e(S') + 1$.
  \end{enumerate}

  This exhausts all possibilities in the case $d(S)=1$: 
  at least one of the components $S_\Horz$ or $S_\Vert$ is a disk,
  and so splitting separates at least one of them.
  
  \textbf{Case (3)} \enspace
  In the case $d(S) = 1$ we already know that $[S] \equiv 0 \mod{2^{d+1}}$.
  If $d(S) \ge 2$ then the four surfaces $S_1,S_2,S_3,S_4$
  satisfy $d(S_i) \ge d(S)-1$, whence $[S_i] \equiv 0 \mod{2^d}$.
  Equation \eqref{eq:Determinant} then implies 
  that $[S] \equiv [S'] \mod{2^{d+1}}$.
\end{proof}

\subsection{Application to ribbon links} \label{sub:RibbonLinks}

For a ribbon knot $K = \partial S$, Proposition \ref{prop:BandCrossingChange} 
says that $\det(K)$ is invariant under band crossing changes of $S$.
This is a well-known property for the classical determinant:
even the Alexander--Conway polynomial $\Delta(K)$ does 
not change (see Remark \ref{rem:AlexanderBandCrossing}).
This observation trivially holds for ribbon links 
with $n \ge 2$ components, for which we always have $\Delta(L)=0$.
The point of Proposition \ref{prop:BandCrossingChange} is that 
after dividing out the factor $V(\bigcirc^n)$ in $V(L)$ 
we obtain the desired property for the Jones determinant: 

\begin{corollary} 
  \label{cor:JonesDeterminant}
  Consider an $n$--component ribbon link $L$ bounding 
  a collection of ribbon disks $S \subset \R^3$.
  Then the Jones nullity is $\Null V(L) = n-1$, and the determinant 
  $\det V(L) = [S]$ is invariant under band crossing changes.
\end{corollary}

\begin{proof}
  We know from Corollary \ref{cor:NullityBound} that 
  $\Null V(L) \le n-1$ for all $n$--component links.
  According to Proposition \ref{prop:EulerJonesNullity} 
  we have $\Null V(L) \ge n-1$ for $n$--component ribbon links.
  We conclude that $\Null V(L) = n-1$ and so 
  $\det V(L) = [V(L)/V(\bigcirc^n)]_{(q=i)} = [S]$.
  Proposition \ref{prop:BandCrossingChange} ensures that 
  $\det V(L)$ is invariant under band crossing changes.
\end{proof}

\begin{corollary} 
  \label{cor:JonesDetMult}
  Consider an $n$--component ribbon link $L = K_1 \cup\dots\cup K_n$
  that bounds a collection of ribbon disks $S \subset \R^3$ 
  without  mixed ribbon singularities, which means that 
  distinct disks never intersect each other.
  Then the Jones determinant satisfies $\det V(L) = \det(K_1) \cdots \det(K_n)$ 
  and is thus a square integer. 
\end{corollary}

\begin{proof}
  Since there are no mixed ribbon singularities, we can change 
  band crossings from $S = S_1 \cup\dots\cup S_n$ to 
  $S' = S_1 \sqcup\dots\sqcup S_n$.  Using the invariance 
  established in Corollary \ref{cor:JonesDeterminant},
  we conclude that $\det V(L) = \det V(L') = \det(K_1) \cdots \det(K_n)$.
\end{proof}

\begin{remark} \label{rem:CongruenceMod16}
  If we allow ribbon disks to intersect each other,
  then multiplicativity holds at least modulo $16$:
  for mixed ribbon singularities, the proof of case (1a) 
  of Lemma \ref{lem:Congruence} shows that $[S] \equiv [S'] \mod{16}$
  holds in Equation \eqref{eq:Determinant}.
  Having replaced all mixed ribbon singularities by ribbon crossings,
  we can apply Corollary \ref{cor:JonesDetMult} to conclude that
  $\det V(L) \equiv \det V(L') = \det(K_1) \cdots \det(K_n)$,
  so in particular $\det V(L) \equiv 1 \mod{8}$.
  We have to work a bit harder to improve this congruence from $16$ to $32$, 
  which is where the full details of Lemma \ref{lem:Congruence} come into play.
\end{remark}

\begin{theorem} \label{thm:CongruenceMod32}
  Consider an $n$--component ribbon link $L = K_1 \cup\dots\cup K_n$,
  bounding a collection of ribbon disks $S \subset \R^3$.  
  Suppose that in Equation \eqref{eq:Determinant} the depicted ribbon 
  singularity involves two distinct disks, $S_\Horz \ne S_\Vert$.  
  Then $[S_1]-[S_2] \equiv [S_3]-[S_4] \mod{16}$ and thus $[S] \equiv [S'] \mod{32}$.
\end{theorem}

\begin{proof} 
  We proceed by a double induction.
  The first induction is on the ribbon number $r(S)$. 
  If $r(S) = 1$, then all links are trivial, i.e., 
  $L = L' = L_1 = L_2 = L_3 = L_4 = \bigcirc^n$,
  and so $[S] = [S'] = [S_1] = [S_2] = [S_3] = [S_4] = 1$.

  If $r(S) \ge 2$, we proceed by induction on the number $k$ 
  of mixed singularities of $S_\Horz$. 
  If $k=1$ then $S_\Horz$ is not involved in any other
  mixed ribbon singularity besides the shown one.
  Applying band crossing changes (Corollary \ref{cor:JonesDeterminant})
  we can achieve that $S_\Horz$ lies above all other components,
  except of course at the shown ribbon singularity.
  This situation is depicted in \fullref{fig:Mutation}:
  $S_1$ and $S_3$ are two connected sums,
  while $S_2$ and $S_4$ are mutants modulo some 
  band twisting (see Proposition \ref{prop:TwistInvariance}).
  This implies that $[S_1]=[S_3]$ and $[S_2]=[S_4]$.
  The difference $[S_1]-[S_2] = [S_3]-[S_4]$ is a multiple 
  of $8$ according to Remark \ref{rem:CongruenceMod16}.
  Equation \eqref{eq:Determinant} then implies that $[S] \equiv [S'] \mod{32}$.
  \newcommand{\pic}[1]{\raisebox{-5.5ex}{\includegraphics[height=12ex]{#1}}}
  \begin{figure}[h]
    \begin{xalignat*}{3}
      S & = \pic{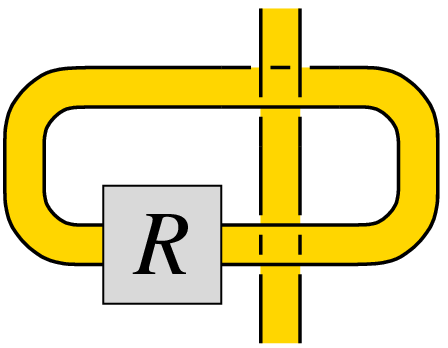} 
      & S_1 & = \pic{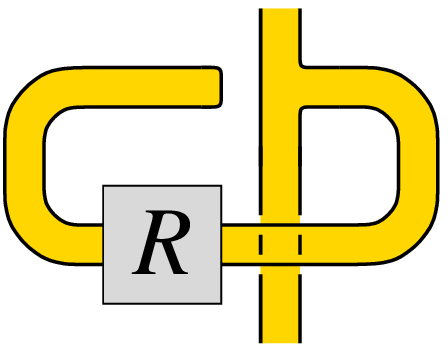} 
      & S_2 & = \pic{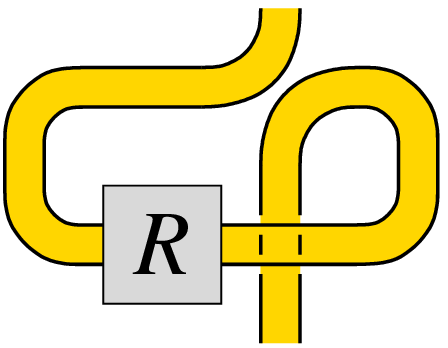} = \pic{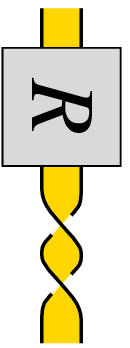}
      \\
      S' & = \pic{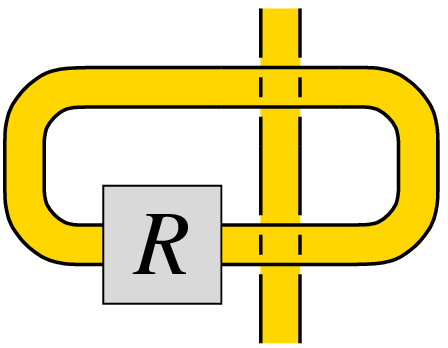}
      & S_3 & = \pic{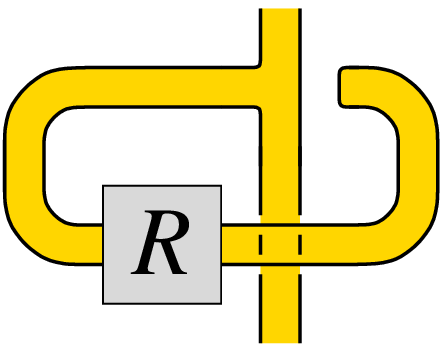} 
      & S_4 & = \pic{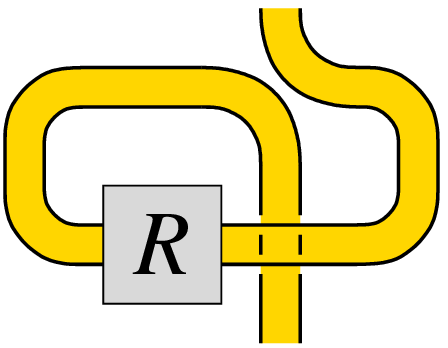} = \pic{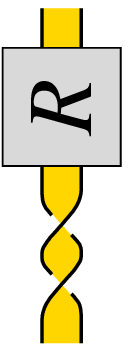}
    \end{xalignat*}
    \caption{Resolving the only ribbon singularity between $S_\Horz$ and $S_\Vert$}
    \label{fig:Mutation}
  \end{figure}
    
  Suppose next that $k\ge2$, that is, $S_\Horz$ is involved in 
  another mixed ribbon singularity.  By induction, it suffices 
  to replace one such ribbon singularity by a band crossing.
  This replacement translates $S,S',S_1,S_2,S_3,S_4$ to six new 
  diagrams $\bar{S}, \bar{S}', \bar{S}_1, \bar{S}_2, \bar{S}_3, \bar{S}_4$,
  each having one less ribbon singularity.  By induction 
  we know that $[\bar{S}] \equiv [\bar{S}'] \mod 32$.

  If our second ribbon singularity involves $S_\Horz$ 
  and some third component different from $S_\Vert$, 
  then we can apply Remark \ref{rem:CongruenceMod16} to all four diagrams 
  on the right hand side to obtain the congruence 
  $[S_1]-[\bar{S}_1] \equiv [S_2]-[\bar{S}_2] \equiv
  [S_3]-[\bar{S}_3] \equiv [S_4]-[\bar{S}_4] \equiv 0 \mod 16$.
  Equation \eqref{eq:Determinant} then implies 
  that $[S] \equiv [S'] \mod{32}$.

  \begin{figure}[h]
    \centering
    \includegraphics[height=28ex]{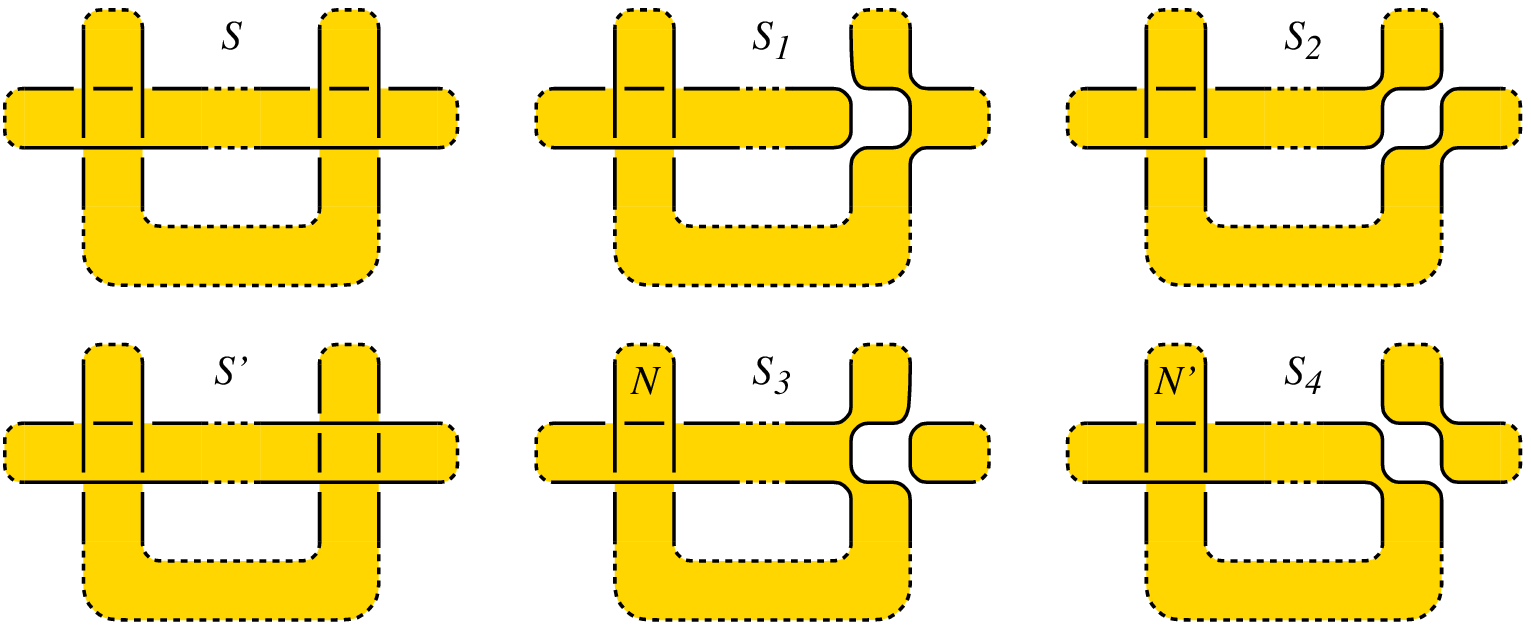}
    \caption{Resolving a second ribbon singularity between $S_\Horz$ and $S_\Vert$}
    \label{fig:RibbonResolveSubtlety2}
  \end{figure}
  
  The only problem arises when our second ribbon singularity involves 
  both $S_\Horz$ and $S_\Vert$.  Suppose for example that the western 
  and southern pieces of $S$ meet again in a second ribbon singularity, 
  as depicted in \fullref{fig:RibbonResolveSubtlety2}.
  Then this is still a mixed singularity in $S_1$ and $S_2$, 
  and so $[S_1]-[\bar{S}_1] \equiv [S_2]-[\bar{S}_2] = 0 \mod 16$
  by Remark \ref{rem:CongruenceMod16}.  
  But in $S_3$ and $S_4$ it becomes a pure singularity.
  Remark \ref{rem:DeltaMod16} in the proof of Lemma \ref{lem:Congruence} case (1b) 
  shows that $[S_3]-[\bar{S}_3] \equiv [S_4]-[\bar{S}_4] \equiv 8\Delta \mod{16}$:
  the northern pieces $N$ and $N'$ are pierced by the newly formed 
  annulus in exactly the same essential singularities.  
  We conclude that 
  \begin{align*}
    [S] - [S'] 
    & \equiv 2 \bigl( [S_1] - [S_2] + [S_3] - [S_4] \bigr)
    \\
    & \equiv 2 \bigl( [\bar{S}_1] - [\bar{S}_2] + [\bar{S}_3] - [\bar{S}_4] \bigr)
    \equiv [\bar{S}] - [\bar{S}'] \equiv 0 \mod 32
  \end{align*}
  because $[S_1] \equiv [\bar{S}_1]$ and $[S_2] \equiv [\bar{S}_2]$ 
  and $[S_3]-[S_4] \equiv [\bar{S}_3]-[\bar{S}_4]$ modulo $16$.
\end{proof}

\begin{corollary}[general multiplicativity modulo $32$] \label{cor:MultiplicativityMod32}
  Every $n$--component ribbon link $L = K_1 \cup\dots\cup K_n$ satisfies 
  the congruence $\det V(L) \equiv \det(K_1) \cdots \det(K_n) \mod{32}$. 
\end{corollary}

\begin{proof}
  We first replace all mixed ribbon singularities by ribbon crossings:
  Theorem \ref{thm:CongruenceMod32} ensures that
  $\det V(L) \equiv \det V(L') \mod{32}$.
  We can then apply Corollary \ref{cor:JonesDetMult}.
\end{proof}

\begin{example} \label{exm:32IsOptimal}
  The value $32$ is best possible: the $2$--component link $L = 10n36$ 
  depicted below is ribbon, whence $\det(L) = 0$,
  and its Jones polynomial factors as 
  \[ 
  V(L) = (q^+ + q^-)
  (-q^{+8} + 2q^{+6} - 3q^{+4} + 4q^{+2} - 3 + 4q^{-2} - 3 q^{-4} + 2q^{-6} - q^{-8}).
  \]
  Here we find $\det V(L) = -23$ whereas the components 
  satisfy $\det(K_1) = 1$ and $\det(K_2) = 9$.
  The congruence $-23 \equiv 9 \mod{32}$ is satisfied, and $32$ is optimal.
  \[
  \newcommand{\pic}[1]{\raisebox{-7ex}{\includegraphics[height=15ex]{#1}}}
  L = \pic{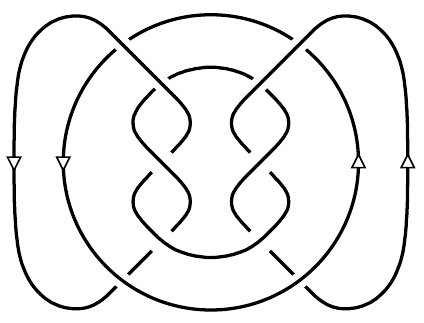}
  \qquad
  L' = \pic{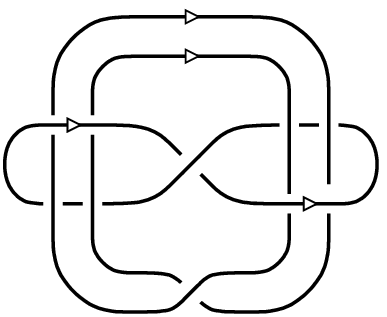}
  \]
\end{example}

\begin{example}
  Not all links with maximal nullity $\Null V(L) = n-1$ 
  satisfy multiplicativity modulo $32$.  
  For $L' = 10n57$, for example, we find $\det(L') = 0$ and 
  \[
  V(L') = (q^+ + q^-) (q^{+6} - 2q^{+4} + 2q^{+2} - 2 + 3q^{-2} - 2q^{-4} + 2q^{-6} - q^{-8})
  \]
  whence $\det V(L') = -15$.  Both components separately are trivial,
  and $\det V(L') \equiv 1$ holds modulo $16$ but not modulo $32$.  
  (In particular, $L'$ is not a ribbon link.  This is no surprise:
  determinant and signature vanish but the Alexander polynomial does not.)
\end{example}

\subsection{Satellites of ribbon knots} \label{sub:RibbonSatellites}

Our results contain information for links with two 
or more components, but at first sight they seem void for knots. 
One possible application is via the construction of \emph{satellites}:

Every oriented knot $K \subset \R^3$ can be equipped with a tubular neighbourhood, 
that is, an embedded torus $f \co \D^2 \times \S^1 \into \R^3$, $f(s,t) = f_s(t)$,
such that $f_0$ parametrizes $K$ satisfying $\lk(f_0,f_1) = 0$ 
and $\lk( f_0, f|_{\S^1{\times}\{1\}} ) = +1$.
Such an embedding $f$ exists and is unique up to isotopy.
For a link $P \subset \D^2 \times \S^1$, the image
$f(P) \subset \R^3$ is called the \emph{satellite} 
of $K$ with \emph{pattern} $P$, and will be denoted by $K \ast P$.

\begin{definition}
  We say that $P \subset \D^2 \times \S^1$ is a \emph{ribbon pattern}
  if $\bigcirc \ast P$ is a ribbon link, where $\bigcirc$ denotes 
  the trivial knot.  This means that the standard (unknotted and untwisted) 
  embedding of the torus $\D^2 \times \S^1 \into \R^3$ 
  maps $P$ to a ribbon link in $\R^3$.
\end{definition}


\begin{proposition}
  If $K$ is a ribbon knot and $P = P_1 \cup\dots\cup P_n$ 
  is an $n$--component ribbon pattern, then the satellite 
  $K \ast P$ is an $n$--component ribbon link.
  \qed
\end{proposition}

\begin{Notes}
\begin{proof}
  Let $g \co \D^2 \looparrowright \R^3$ be a smooth immersion
  that parametrizes a ribbon surface for our knot $K$ 
  parametrized by $\kappa = g|_{\S^1} \co \S^1 \into \R^3$.
  The map $g$ extends to a thickened ribbon surface, i.e., a smooth map 
  $\hat{g} \co [0,1] \times \D^2 \looparrowright \R^3$
  satisfying $\hat{g}|_{\{0\} \times \D^2} = g$ and behaving 
  in the standard way around ribbon singularities.
  Given $c \in \N$ we define a $c$--component ribbon surface
  $g^{\cable{c}} \co \{1,\dots,c\} \times \D^2 \looparrowright \R^3$
  by setting $g^{\cable{c}}(k,t) := \hat{g}(\frac{k-1}{c},t)$.  

  The restriction $f = \hat{g}|_{[0,1] \times \S^1}$ is a framing 
  of $\kappa$, and its construction ensures $\lk(f) = 0$.  
  The restriction of the ribbon surface immersion 
  $g^{\cable{c}} \co \{1,\dots,c\} \times \D^2 \looparrowright \R^3$ to the boundary 
  $\partial g^{\cable{c}} \co \{1,\dots,c\} \times \S^1 \into \R^3$
  thus yields the $0$-framed $c$-cable $K^{\cable{c}}$. 
  This proves that $K^{\cable{c}}$ is a ribbon link, as claimed,
  by providing a ribbon surface by explicit construction.
\end{proof}
\end{Notes}

\begin{remark}
  Starting with a ribbon pattern $P$, the satellite $K \ast P$ may be ribbon 
  even though $K$ is not; see Rolfsen \cite[Example 8E33]{Rolfsen:1990}.
\end{remark}

\begin{corollary}
  If $K$ is a ribbon knot, then for every $c \in \N$
  the $0$-framed $c$-cable $K^{\cable{c}}$ is a ribbon link,
  whence $\Null V(K^{\cable{c}}) = n-1$ and 
  $\det V(K^{\cable{c}}) \equiv \det(K)^c \mod{32}$.
  \qed
\end{corollary}

\begin{example}
  The knot $K = 6_1$ is the smallest ribbon knot; it has determinant $\det(K) = 9$.
  The Jones determinant of its two-cable is $\det V(K^2) = 49 = 9^2 - 32$.
  For the three-cable we find $\det V(K^3) = 1785 = 9^3 + 33 \cdot 32$.
  Again $32$ is best possible.  
\end{example}

This corollary is quite pleasant, yet it does not seem to obstruct ribbonness.
A possible explanation is that every cable $K^{\cable{c}}$ is a boundary link: 
Question \ref{quest:BoundaryLinks} below asks whether this entails the same 
algebraic consequences, even if the initial knot $K$ is not ribbon. 

\begin{Notes}
\begin{remark}
  Analogously we can define the $0$-framed cabling
  $L^{\cable{c}} = L_1^{\cable{c_1}} \cup \dots \cup L_n^{\cable{c_n}}$
  for every link $L = L_1 \cup \dots \cup L_n \subset \R^3$
  and cabling numbers $c = (c_1,\dots,c_n) \in \N^n$.
  If $L$ is a ribbon link with $n$ components and $c \in \N^n$,
  then the $0$-framed $c$-cable $L^{\cable{c}}$ is a ribbon link with 
  $|c| = c_1 + \dots + c_n$ components.  As a consequence,
  its Jones polynomial $V(L^{\cable{c}})$ is divisible by $V(\bigcirc^{|c|})$ and 
  $\det V(L_1^{\cable{c_1}} \cup \dots \cup L_n^{\cable{c_n}}) 
  \equiv \det(L_1)^{c_1} \cdots \det(L_n)^{c_n} \mod{32}$.
\end{remark}

\begin{remark}
  Notice that for \emph{every} knot $K$ we have $\Null K^{\cable{c}} = c-1$:
  given a Seifert surface $S \subset \R^3$ with $\partial S = K$, we can thicken
  $S$ and construct $c$ parallel Seifert surfaces for the components of $K^{\cable{c}}$.  
  Making this surface connected by gluing $c-1$ tubes adds $c-1$ elements to the basis 
  of $H_1(S^{\cable{c}})$ and creates a subspace on which the Seifert form vanishes identically.
  If it should turn out that Seifert nullity and Jones nullity co\"incide,
  then the preceding criterion adds nothing new to the existing tools.
\end{remark}
\end{Notes}


\section{Open questions and perspectives} \label{sec:OpenQuestions}

Our results can be seen as a first step towards 
understanding the Jones polynomial of ribbon links. 
They suggest further questions and generalizations in several directions.

\subsection{From ribbon to slice}

At the time of writing it is not known
whether every smoothly slice link is a ribbon link.
Our results thus offer two perspectives: 
either they extend from ribbon to smoothly slice links, 
which would be rather satisfactory for the sake of completeness.
Or, even more interestingly, there exist smoothly slice links 
for which some (suitably refined) ribbon criteria fail: 
this would refute the long-standing conjecture
``smoothly slice implies ribbon'' conjecture, at least for links. 

\begin{question} \label{quest:SliceRibbon}
  Do Theorems \ref{Thm:RibbonJonesNullity} and \ref{Thm:MultiplicativityMod32} 
  generalize from ribbon links to slice links?
\end{question}

Quite possibly our results hold true in this generalized setting.
An elegant way to show this would be to extend an observation 
of Casson, recorded by Livingston \cite[\textsection 2.1]{Livingston:2005}:  
for every slice knot $K$ there is a ribbon knot $K'$ 
such that their connected sum $K \cs K'$ is ribbon.
Is there an analogous trick for slice links?

A negative answer to Question \ref{quest:SliceRibbon} 
would be spectacular, but it remains to be examined 
whether the Jones polynomial can detect 
such subtle differences, if at all they exist.
As Livingston \cite[\textsection 10, Problem 1]{Livingston:2005} put it:
``One has little basis to conjecture here.  
Perhaps obstructions will arise (...) but 
the lack of potential examples is discouraging.''

\subsection{From Jones to {\sc Homflypt}}

It is tempting to generalize \fullref{Thm:RibbonJonesNullity} 
to other knot polynomials, in particular to the \textsc{Homflypt} 
polynomial, or at least to $\V$ for $N$ prime:

\begin{question}
  Does \fullref{Thm:RibbonJonesNullity} extend 
  to the generalized Jones polynomial 
  in the sense that $\V(L) = \V(\bigcirc^n) \cdot \rV(L)$
  for every ribbon link $L$?
\end{question}

This holds for $N=0$ because the Alexander--Conway polynomial vanishes for $n \ge 2$.
The case $N=1$ is trivial.  \fullref{Thm:RibbonJonesNullity} 
settles the case $N=2$.  The question for $N\ge3$ is open, but 
sample calculations suggest that the factorization seems to hold.   

\begin{remark}
  The Kauffman bracket has served us well in the inductive proof for $N=2$.
  For $\V$ with $N \ge 2$, Murakami--Ohtsuki--Yamada \cite{MOY:1998}
  have developed an analogous oriented state model.  Even though the approach 
  is very similar, the calculations generalizing \sref{sec:RibbonJonesNullity} 
  get stuck because certain terms do not cancel each other.
  This makes the argument harder and some additional ideas will be needed.
\end{remark}

\begin{question}
  How can \fullref{Thm:RibbonJonesNullity} be generalized 
  to the Kauffman polynomial \cite{Kauffman:1990b}?
  The obvious generalization is false:
  the Kauffman polynomial $F(L) \in \Z[a^\pm,z^\pm]$
  of the two-component ribbon link $L = 10n36$, 
  for example, is not divisible by $F(\bigcirc^2)$.
\end{question}

\subsection{Towards Khovanov homology}

The most fertile development in the geometric understanding 
and application of the Jones polynomial in recent years 
has been Khovanov homology \cite{Khovanov:2000,BarNatan:2005}.
Applying the philosophy of categorification to the Kauffman bracket,
this theory associates to each link $L$ a bigraded homology
$\Kh(L) = \smash{\bigoplus_{i,j \in \Z} \Kh_{i,j}(L)}$ as an invariant.  The polynomial 
$P(t,q) = \smash{ \sum_{i,j \in \Z} \, t^i \, q^j \, \dim_\Q( \Kh_{i,j}(L) \otimes \Q ) }$ 
is an invariant of $L$ that specializes for $t = -1$ to
the Jones polynomial, $P(-1,q) = (q^+ + q^-) \cdot V(L)$.

\begin{question}
  What is the homological version of 
  $V(L) = V(\bigcirc^n) \cdot \tilde{V}(L)$?
\end{question}

The na\"ive generalization 
would be $\Kh(L) \cong \Kh(\bigcirc^n) \otimes \tilde{\Kh}(L)$.
The first problem in stating and proving a result of this type is that 
the isomorphism must be made explicit and should be as natural as possible.
A polynomial factorization such as $P(L) = (q^+ + q^-)^n \cdot \tilde{P}(L)$ 
is a weaker consequence that does not require isomorphisms in its statement.
Sample calculations, say for $L = 10n36$, show that these simple-minded
factorizations do not hold, neither over $\Q$ nor over $\Z/2$.  

Since $P(-1,t)$ can be seen as the graded Euler characteristic
of $\Kh(L)$, another analogy could prove useful: for every fibration
$p \co E \to B$ with fibre $F$, the Leray--Serre spectral sequence
with $E^2_{p,q} = H_p( B, H_q(F) )$ converges to $H_{p+q}(E)$, 
whence $\chi(E) = \chi(B) \cdot \chi(F)$.
Can the factorization $V(L) = V(\bigcirc^n) \cdot \tilde{V}(L)$ 
be derived as the Euler characteristic of some spectral sequence?
What is the r\^ole of the factor $\tilde{V}(L)$?

\subsection{Ribbon cobordism}

On top of the quantitative improvement of a more detailed numerical invariant
$P(t,q)$, Khovanov homology provides an important qualitative improvement:
it is functorial with respect to link cobordism
(Jacobsson \cite{Jacobsson:2004}, Khovanov \cite{Khovanov:2006}).
In this vein Rasmussen \cite{Rasmussen:SliceGenus} established 
a lower bound for the slice genus of knots, providing a new proof of 
the Milnor conjecture on the unknotting number of torus knots.
It thus seems reasonable to hope that $\Kh(L)$ captures 
more subtle properties of slice and ribbon links.

\begin{question}
  Is there a functorial version of Theorems 
  \ref{Thm:RibbonJonesNullity} and \ref{Thm:MultiplicativityMod32}?
\end{question}

Gordon \cite{Gordon:1981} introduced the notion of \emph{ribbon concordance}.
In the slightly more general setting of \fullref{Prop:EulerJonesNullity} 
we consider a link $L \subset \R^3$ that bounds a properly embedded smooth surface 
$S \subset \R^4_+$ of positive Euler characteristic $n \ge 1$ and without local minima.  
Cutting out small disks around $n$ local maxima we obtain a \emph{ribbon cobordism}
$C \subset \R^3 \times [0,1]$ from $L = C \cap (\R^3\times\{0\})$ 
to $\bigcirc^n = C \cap (\R^3\times\{1\})$ such that $\chi(C) = 0$. 
This induces homomorphisms $c \co \Kh(L) \to \Kh(\bigcirc^n)$
and $c^* \co \Kh(\bigcirc^n) \to \Kh(L)$. 

\begin{question}
  Is $c$ surjective?  Is $c^*$ injective?
  Better still, do we have $c \circ c^* = \id$? 
  A positive answer would exhibit $\Kh(\bigcirc^n)$ 
  as a direct summand of $\Kh(L)$.
\end{question}

C Blanchet suggested that the chain complex $\CKh(L)$
could be considered as a module over $\CKh(\bigcirc^n) = A^{\otimes n}$,
where $A$ is the Frobenius algebra used in Khovanov's construction.
This leads to the natural question: when is $\CKh(L)$ essentially free over $\CKh(\bigcirc^n)$?
A positive answer would explain the factorization $V(L) = V(\bigcirc^n) \cdot \tilde{V}(L)$ 
and potentially give some meaning to the reduced Jones polynomial $\tilde{V}(L)$.

\subsection{Other geometric criteria} \label{sub:OtherGeometricCriteria}

We have concentrated here on ribbon links,
but many other links $L$ may also satisfy 
the conclusion of Theorems \ref{Thm:RibbonJonesNullity}
and \ref{Thm:MultiplicativityMod32}:

\begin{question} \label{quest:BoundaryLinks}
  Which other geometric properties imply that 
  $V(L)$ is divisible by $V(\bigcirc^n)$?
  Do they imply that $L = K_1 \cup\dots\cup K_n$ satisfies 
  $\det V(L) \equiv \det(K_1) \cdots \det(K_n)$ modulo $32$?
  More concretely: does this hold for boundary links?
\end{question}

We recall that an $n$--component link $L = L_1\cup\dots\cup L_n$ 
is a \emph{boundary link} if it bounds a surface $S = S_1 \cup\dots\cup S_n$
embedded in $\R^3$ such that $\partial S_i = L_i$ for each $i=1,\dots,n$.
(We can always find a connected surface $S$ such that $\partial S = L$,
but here we require that $S$ consist of $n$ disjoint surfaces $S_1,\dots,S_n$.)
The Seifert nullity of a boundary link is maximal, perhaps its Jones nullity too.
It is certainly not enough that pairwise linking numbers 
vanish: the Whitehead link $W$ satisfies $\lk(W) = 0$ but $\det(W) = 8i$.


\begin{question} \label{quest:ClassicalNullity}
  For which links $L$ do we have equality $\Null_\omega(L) = \Null_\omega \V(L)$?
\end{question}

The following observations show that this question is not completely absurd:
\begin{itemize}
\item
  Equality holds for all knots $K$ and prime $N$,
  because $\Null_\omega(K) = \Null_\omega \V(K) = 0$, that is,
  $\det_\omega(K) = \V(L)|_{(q \mapsto \omega)}$ is always non-zero.
\item
  Equality also holds for all two-component links and prime $N$, 
  because we have $\Null_\omega(L) \in \{0,1\}$ and 
  $\Null_\omega \V(L) \in \{0,1\}$, as well as
  $\det_\omega(L) = \V(L)|_{(q \mapsto \omega)}$.
\item
  Equality is preserved under disjoint union, connected sum,
  mirror images, and reversal of orientations. 
\item
  \fullref{Thm:RibbonJonesNullity} ensures that,
  at least for $N=2$, equality holds for all ribbon links.
\end{itemize}  

Garoufalidis \cite[Corollary 1.5]{Garoufalidis:1999}
observed that $\Null(L) \ge 4$ implies $\Null V(L) \ge 2$.
This follows from Equation \eqref{eq:Garoufalidis} and
a result of Lescop \cite[\textsection 5.3]{Lescop:1996} saying that
$\Lambda(M)$ vanishes for every manifold with 
$\dim H_1(M,\Q) \ge 4$.  In the special case $M = \Sigma^2_L$ 
this can possibly be sharpened to show that 
$\Null(L) \ge 2$ implies $\Null V(L) \ge 2$.

\begin{question} \label{quest:Concordance}
  Does link concordance $L \sim L'$ imply that $\Null_\omega \V(L) = \Null_\omega \V(L')$?
  If so, which congruence holds between $\det_\omega \V(L)$ and $\det_\omega \V(L')$?
\end{question}

For the Alexander--Conway polynomial the corresponding questions 
were answered by Kawauchi \cite{Kawauchi:1978} and Cochran \cite{Cochran:1985}.
Equality in Question \ref{quest:ClassicalNullity} would imply
concordance invariance of $\Null_\omega \V(L)$, 
because the Seifert nullity is a concordance invariant.

\subsection{Does the Jones polynomial determine the signature mod $4$?}

The determinant $\det(L)$ and the signature $\Sign(L)$ 
of a link $L$ are related by the formula 
\begin{equation} \label{eq:DetSign}
  \det(L) = i^{-\Sign(L)} \cdot |\det(L)| . 
\end{equation}
Conway \cite{Conway:1969} used this together with 
$\newcommand{\skein}[1]{\Sign(\raisebox{-0.4ex}{\includegraphics[height=2ex]{#1}})}
\skein{skein+} - \skein{skein-} \in \{0,1,2\}$ 
to calculate signatures recursively.  An analogous formula holds 
for every $\omega \in \S^1$ with $\mathrm{im}(\omega) > 0$.

If $\omega$ is a primitive $2N$th root of unity, 
we know that $\det_\omega(L) \ne 0$ at least for knots.
For links with $n \ge 2$ components Conway's signature calculation 
is obstructed by the fact that the determinant may vanish, 
in which case Equation \eqref{eq:DetSign} contains no information.
This happens exactly when $\Null(L) \ge 1$.
One might suspect that the stronger condition
$\det V(L) = i^{-\Sign(L)} \cdot |\det V(L)|$ holds.
Unfortunately this is \emph{false} in general:
see Example \ref{exm:32IsOptimal} above
for a ribbon link with $\det V(L) < 0$.

The formula thus needs some correction. 
Of course we can \emph{define} a link invariant 
$\varepsilon(L) \co \Links \to \{\pm 1, \pm i\}$ by
$\varepsilon(L) := i^{\Sign(L)} \cdot \det V(L) / |\det V(L)|$.
The topological meaning of this factor $\varepsilon(L)$, however, is not clear.  
It is also unknown whether $\varepsilon(L)$ 
can be deduced from the Jones polynomial alone.  
If so, then the Jones polynomial would determine the signature 
of all \emph{links} via Conway's skein theoretic recursion.

\subsection{Surface invariants of finite type}

\fullref{sec:BandCrossingChanges} introduces and illustrates the concept 
of surface invariants that are of finite type with respect to band crossing changes.
This is an interesting analogy and extension of link invariants of finite type.
What is the precise relationship between these two classes of invariants?
In our examples the surface invariant $S \mapsto d_k(\partial S)$ only depends 
on the boundary of $S$, but in general this need not be the case.
Can we generate more non-trivial examples from the \textsc{Homflypt} 
or the Kauffman polynomial or other quantum invariants?
What is their geometric significance?


The general finite type approach to surfaces will be 
the object of a forthcoming article \cite{Eisermann:SurfInv}.  
Generalizing \fullref{sec:BandCrossingChanges}, one proceeds as follows:
\begin{itemize}
\item Introduce the filtration induced by band crossing changes and band twists. 
\item Study the graded quotients and extract combinatorial data modulo relations.
\item Integrate (in low degree at least) combinatorial data to invariants of surfaces.
\end{itemize}

It is interesting to note that the Euler characteristic 
of the surface intervenes in a natural and non-trivial way.
Two perspectives seem to be most promising: 
Considering \emph{immersed} ribbon surfaces one might 
hope to derive lower bounds for the ribbon genus.
Turning to \emph{embedded} surfaces one might try 
to reconcile the classical approach of Seifert surfaces
with finite type invariants.  Here Vassiliev--Goussarov invariants 
are known to be too restrictive, see Murakami--Ohtsuki \cite{MurakamiOhtsuki:2001}.

In analogy with the tangle category modelling knots and links,
one can construct a category modelling ribbon surfaces.
Once we have a presentation of this category by generators and relations,
we can look for representations and extract invariants.
Quite plausibly some of the extensively studied quantum representations 
extend to this setting, and the introduction of surfaces
might reveal more topological features.


\bibliographystyle{gtart}

\bibliography{ribbonlinks}

\end{document}